\documentclass[11pt]{amsart}
\usepackage{comment}
\usepackage{amsmath,amsfonts,amssymb,amstext,amsthm,graphicx,fancyhdr,mathtools,epstopdf}
\usepackage{color}
\usepackage[top=1in, bottom=1in, left=1in, right=1in]{geometry}
\usepackage[titletoc,title]{appendix}
\usepackage{bbm}
\usepackage{bm}
\usepackage{dsfont} 
\usepackage[english]{babel}
\usepackage[T1]{fontenc}
\usepackage{latexsym}
\usepackage{mathrsfs}
\usepackage{graphics}
\usepackage{caption}
\usepackage{subcaption}
\usepackage{graphicx}
\usepackage{soul}
\usepackage{enumerate, paralist, enumitem} 
\usepackage{color}
\usepackage{fullpage}
\usepackage{hyperref}
\usepackage[all,cmtip]{xy}
\usepackage{soul}

\hypersetup{colorlinks=true}
\DeclareFontFamily{OT1}{pzc}{}
\DeclareFontShape{OT1}{pzc}{m}{it}{<-> s * [1.10] pzcmi7t}{}
\DeclareMathAlphabet{\mathpzc}{OT1}{pzc}{m}{it}

\newtheorem{theorem}{Theorem}[section]
\newtheorem*{theorem*}{Theorem}
\newtheorem{lemma}{Lemma}[section]
\newtheorem{corollary}{Corollary}[section]
\newtheorem{proposition}{Proposition}[section]
\newtheorem*{proposition*}{Proposition}
   \newtheoremstyle{example}{\topsep}{\topsep}%
     {}
     {}
     {\bfseries}
     {}
     {\newline}
     {\thmname{#1}\thmnumber{ #2}\thmnote{ #3}}

\theoremstyle{definition}

\newtheorem*{assmp*}{Assumption}

\theoremstyle{remark}
\newtheorem{remark}{Remark}[section]

\newcommand{\E}{\mathbb{E}}
\newcommand{\N}{\mathbb{N}}
\newcommand{\R}{\mathbb{R}}
\newcommand{\Prob}{\mathbb{P}}
\newcommand{\Poly}{\mathscr{P}} 

\newcommand{\h}{\hslash}

\newcommand{\calJ}{\mathds{J}} 
\newcommand{\calT}{\mathcal{T}} 
\newcommand{\calJd}{\mathbf{J}}

\newcommand{\mo}{{\mu}}

\newcommand{\lam}{{{\lambda}}} 
\newcommand{\vat}{{\mathsf{r}_1}} 

\renewcommand{\varphi}{{\phi_{\vat}^{\checkmark}}}

\newcommand{\J}{\mathds{Q}} 
\newcommand{\Q}{\mathbf{Q}}

\newcommand{\gab}[1][\lam,\mo]{{\beta_{#1}}}

\newcommand{\bpsi}{\mathscr{\beta}} 

\newcommand{\Leb}{\mathrm{L}} 

\renewcommand{\leq}{\leqslant} 
\renewcommand{\geq}{\geqslant}

\numberwithin{equation}{section}

\author{Giuseppe D'Onofrio}
\address{Dipartimento di Scienze Matematiche, Politecnico di Torino, 10129 Turin, Italy}
\email{giuseppe.donofrio@polito.it}

\author{Pierre Patie}
\address{School of Operations Research and Information Engineering, Cornell University, Ithaca, NY 14853.}
\email{pp396@cornell.edu}

\author{Laura Sacerdote}
\address{Dipartimento di Matematica `G. Peano', Universit\`{a} degli Studi di Torino, Via Carlo Alberto 10, 10123 Torino, Italy}
\email{laura.sacerdote@unito.it}


\usepackage{hyperref}

\setlist[enumerate,1]{label=(\arabic*),ref=(\arabic*)}
\setlist[enumerate,2]{label=(\alph*),ref=(\arabic{enumi})(\alph*)}
\setlist[enumerate,3]{label=(\roman*),ref=(\arabic{enumi})(\alph{enumii})(\roman*)}
\setlist[enumerate,4]{label=(\Alph*),ref=(\arabic{enumi}-\alph{enumii}-\roman{enumiii}-\Alph*)}

\title{Jacobi processes with jumps as neuronal models : a first passage time analysis
}
\begin{document}

\begin{abstract}
To overcome some limits of classical neuronal models, we propose a Markovian generalization of the classical model based on Jacobi processes  by introducing
downwards jumps to describe the activity of a single neuron. The statistical analysis of inter-spike intervals is performed by studying the first-passage times of the proposed Markovian Jacobi process with jumps through a constant boundary. In particular, we  characterize its Laplace transform which is expressed in terms of some generalization of hypergeometric functions that we introduce, and deduce  a closed-form expression for its expectation. Our approach, which is original in the context of first passage time problems, relies on intertwining relations between the semigroups of the classical Jacobi process and its generalization, which have been  recently established in  \cite{CPSV}. A numerical investigation of the firing rate of the considered neuron is performed for some choices of the involved parameters and of the jumps distributions.
\end{abstract}
\maketitle

\section{Introduction}

Among the models used for the description of single neuron's activity
the leaky integrate-and-fire (LIF) model is still an extremely useful tool, despite its age and simplicity \cite{abbott,sac_gir_review}.
The LIF model describes the time evolution of the voltage across the membrane of the neuron until it reaches a certain threshold. This event is called action potential (or spike) and it is believed that
the distribution of these spikes encodes the information that the neurons transfer.
It is assumed that the neuron under study is point-like and receives inputs from the surrounding network of neurons that are summed up ({\em integrate}) producing a change in the voltage value.
The term {\em leaky} indicates that, in the absence of input, the membrane potential decays exponentially to its resting value.
In accordance with the model, the spikes are instantaneous events that are generated as soon as the voltage reaches a certain value for the first time ({\em fire}). After that, the process is reset to its starting
value and the evolution starts over again.
Sometimes a refractory period is added to the model, i.e. there is a  time interval after a spike in which a nerve cell is unable to fire an action potential.

Since for some types of neurons the incoming inputs are frequent and relatively small, a diffusion limit over the discrete process, see  \cite{stein}, describing the membrane potential evolution is performed to gain the higher mathematical tractability of the Ornstein-Uhlenbeck process \cite{sac_OU}. The latter  has been widely used for decades, although it presents some drawbacks.
The Ornstein-Uhlenbeck process, indeed, allows unlimited values for the neuronal potential, it does not include that the changes in the potential of a nerve cell depend on its actual value and it does not take into account the geometry of the neuron.
Some models with multiplicative noise have been proposed  to overcome the first two unrealistic features of the classical LIF model \cite{lindner},\cite{lanska94}.
Among them, recently, a Jacobi process has been proposed for the description of the activity of a neuronal membrane \cite{don_jacobi}.
The Jacobi process has a bounded state space, that is the value of the membrane potential is confined below and above by two fixed values that, for physiological reasons, are called the inhibitory and excitatory reversal potentials. 
 {Moreover, the change in the membrane potential determined by an incoming input depends on the distance between its actual state  and the two reversal potentials,  fact that is well established about the physiology of the synapses, see \cite{eccles}  or \cite{hans_tuc}  for  classical references.}

However, the pure-diffusion models do not account for the spatial geometry of the neurons and do not discriminate among different sources of incoming inputs. 
In fact, more realistic models should assign different weights for the synaptic contribution impinging the neuron in different points of the membrane depending on whether they are more or less close to the trigger zone, as a first attempt to overcome the point-size assumption, in the spirit of multi-compartmental models, see for instance \cite{rod_lan}. In order to include these features in the model, and since
not all the inputs are infinitesimal and their frequencies may  prevent a diffusion limit, jump-diffusion
models have been proposed \cite{giraudo_sacerdote, sirovich}.
These models have been proven to describe the activity of motor neurons  and pyramidal neurons \cite{ditlevsen,melanson_longtin}.
Here, we  investigate the features of a neuronal model
in which the membrane potential evolves, between two consecutive spikes, according to a Jacobi type process with state-dependent downward jumps.

To develop the analysis on its firing activity, we investigate the first-passage-time (FPT) problem for the Jacobi process with jumps, that is the mathematical counterpart of the time of generation of the action potential. To this end, we propose an original approach in this context, which is based on intertwining relationships between the semigroups of the generalized Jacobi processes and the one of the classical Jacobi diffusion process which were identified recently in \cite{CPSV}, see also \cite{miclo_patie} for further analytical results on these semigroups. Intertwining relations are a type of commutation relations that form a classification scheme  for linear operators. They  have proved to be a natural and powerful concept in a variety of contexts in mathematics ranging from the construction of new Markov semigroups to the spectral and ergodicity theory of non-self-adjoint semigroups, see \cite{CPSV, Choi, jar, Patie-Savov-GeL, PSar} and the references therein.  This paper provides an additional application of  such concept in the potential theory of Markov processes by transferring   $q$-invariant functions from a reference semigroup to semigroups that are in its (intertwining) orbit. This device enables us to characterize the Laplace transform of the first passage time of the process through a boundary (or in the modeling framework the firing time of the neuron) in terms of some generalized hypergeometric functions that we introduce, thanks to the relationship with the classical Jacobi model.
We point out that the intertwining approach could also be used to map $q$-invariant functions, or more generally $q$-excessive functions, between semigroups associated to Markov processes with arbitrary jumps.

The strengths of the presented model rely on the improved
adherence to phenomenological reality: the jumps are state-dependent and are able to reduce the firing rate and introduce saturation. The latter feature is observed in other models only if a non-zero refractory period is introduced. Otherwise the firing rate can generally grow unbounded, and this is clearly unrealistic. Moreover the high degree of freedom in the choice of the jump distribution allows the description of different situations.

Finally we stress that, despite the application in the context of mathematical neurosciences, the results on the Jacobi process with jumps and its first passage time through a constant boundary are novel and of a general nature.
We mention that Jacobi processes have been popular in applications such as population genetics, under the name Wright-Fisher diffusion, see e.g. Griffiths et al.~\cite{griffiths:2010,griffiths2018}, Huillet \cite{huillet2007wright}, and Pal~\cite{pal2013}, and in finance, see e.g.~Delbaen and Shirikawa \cite{delbaen2002interest} and Gourieroux and Jasiak \cite{gourieroux2006}.

The paper is organized as follows. In Section 2, we introduce the neuronal model based on the jump-diffusion Jacobi process through its infinitesimal generator and a qualitative description of the dynamics together with  the involved parameters.
Mathematical results on the non-local Jacobi operator, necessary for the analysis of the model, are obtained in Section 3 using intertwining relations between the classical Jacobi semigroup and the non-local one. In particular we show that the process under study has only downward jumps and we provide and explicit form of the Laplace transform of the related first-passage time.
Using these results, in Section 4, an analysis of the firing activity of the neuron described by a Jacobi process with jumps is carried out focusing on some illustrative examples. In particular, we show that the jumps reduce the firing rate and introduce a saturation effect, despite the absence of a refractory period.
We point out that the readers with stronger interest in the biophysical aspect of this work may skip the mathematical details provided in Section 3.

\section{A Jacobi process with jumps as a neuronal model}
\label{section2}

We describe the evolution of the neuronal membrane potential between two consecutive spikes of a single neuron as the Markovian realization $X=(X_t)_{t\geq0}$ of a non-local perturbation 
of the  generator ${\bf J}_V$  of the classical Jacobi neuronal model $V=(V_t)_{t\geq0}$.
We recall that the latter is obtained as a Kurtz-type diffusion approximation of a Stein's model with reversal potentials \cite{stein}. In that model two independent homogeneous
Poisson processes represent the excitatory and inhibitory neuronal inputs, with intensities  $\nu_E>0$ and $\nu_I>0$, respectively. They describe the arrival of excitatory and inhibitory potentials and are such that the input parameters are
\begin{equation}
\label{input_par}
\mu_e=e \nu_E \quad \textrm{ and } \quad \mu_i=i \nu_I
\end{equation}
where $i$ and $e$ are constants such that $ - 1 < i < 0 < e < 1$. Denoting by $V_I < 0 < V_E$
 the inhibitory and excitatory reversal potentials,  respectively, we recall that $V$ is a diffusion on $E_V=[V_I, V_E]$ solving the following stochastic differential equation
\begin{equation}\label{SDE_LIF}
dV_t=\left(-\frac{1}{\tau}V_t+\mu_e(V_E-V_t)+\mu_i(V_t-V_I)\right)dt+\sigma\sqrt{(V_E-V_t)(V_t-V_I)}dW_t.
\end{equation}
 where the diffusion coefficient $\sigma>0$ controls the amplitude of the noise, $W=(W_t)_{t\geq0}$ is a standard Wiener process and $\tau > 0$ is the membrane time constant taking into account
the spontaneous voltage decay  ({\it leak}) toward the resting potential
(set equal to zero here)
in the absence of inputs, $\mu_e$ and $\mu_i$. Finally, the refractory period  is assumed equal to zero.
The model $V=(V_t)_{t\geq0}$ often goes under the name of {\it leaky integrate-and-fire} with reversal potentials \cite{lanska94}.

 Alternatively $V$ can be described through its infinitesimal generator taking the form,  for a smooth function $f$ on $E_V$,
 \begin{equation}
\label{eq:gen_nontransfJ}
{\bf J}_V f(x)=\frac{\sigma^2}{2} (V_E-x)(x-V_I)f''(x)-\left(\frac{1}{\tau} x-\mo_e(V_E-x)-\mo_i(x-V_I)\right)f'(x),
\end{equation}
throughout the paper this way of writing will be more convenient.

A limit of neuronal models based on diffusion processes is that all the inputs which affect the membrane potential are summed together and homogenized disregarding their origin or strength, with the advantage of a continuous trajectory for the dynamics. In \cite{sirovich}, \cite{giraudo_sacerdote}, \cite{tamb_sac_jac}  there are first attempts to introduce jumps but occurring at exponential times. Moreover the mathematical tractability requested the use of jumps of constant amplitude or the use of numerical simulations.
Taking advantage of the intertwining approach here we introduce and study mathematically the case of jumps that are state-dependent both in frequency and amplitude.  
The dynamics of the voltage $X$ between two consecutive spikes of the new class of neuronal models that we propose here is described as a Markov (in fact, a Feller) process on  $E_V$ with c\`adl\`ag trajectories whose infinitesimal generator, for a smooth function $f $ on $ E_V$, is given as the following non-local perturbation of the generator  of the classical Jacobi process
\begin{eqnarray}
\label{eq:gen_nontransf}
\calJ_X f(x) &=& {\bf J}_V f(x)+\int_{V_I}^{V_E} \left(f\left(r\right)-f\left(x\right)\right)N_V(x,dr)
\end{eqnarray}
where the kernel $N_V(x,dr)=\frac{V_E-V_I}{x-V_I}\Pi_V(x,dr)\mathbb{I}_{\{r<x\}}$ with  $\Pi_V$  the measure image, by the mapping $r\mapsto \ln(\frac{x-V_I}{r-V_I})\mathbb{I}_{\{r<x\}}$,  of $\Pi$ a finite non-negative Radon measure on $\R_+$  with $\int_0^\infty r\Pi(dr)< \infty$. We shall show in Proposition \ref{prop:intD} that the family (indexed by $\Pi)$ of linear operators $\calJ_X $ is the infinitesimal generator of a Feller process admitting an unique stationary measure. This will be achieved by identifying a homeomorphism, i.e. an intertwining relation \`a la Dynkin, between these semigroups and the one of Jacobi processes with jumps on $(0,1)$ recently introduced in \cite{CPSV}, where the process jumps from state $x$ to  state $e^{-r}x$ at a frequency given by $\Pi(dr)/x$, that is inversely proportional to the achieved state. 

We point out that the assumptions on the measure $\Pi$ ensure that the operator defined in \eqref{eq:gen_nontransf}, endowed with its domain, is the generator of a Markov process. We also remark that the jumps are only downwards, but, both the amplitude and the intensity of the jumps are state-dependent. In fact if the voltage approaches the inhibitory reversal potential, the number of jumps is high but the corresponding depolarization is small. Conversely for higher values of the voltage the frequency  of jumps  decreases whereas their amplitude  depends on $\Pi$. See Fig. \ref{fig_traj} for an example of a possible path of the proposed model.

If the jump kernel is a finite measure, which is the case of $\Pi_V$ here, there is a nice and more formal path interpretation of the Markov process that can be read off from its generator, see e.g.~Bass \cite{Bass}.
Indeed, one has the following description of dynamics of the voltage $X$: the potential starts by undergoing the same dynamics than the classical Jacobi neuronal model $V$ until being killed at a random time $\calT$ whose survival probability up to time $t$ is given by $e^{-\frac{V_E-V_I }{X_t-V_I}\Pi(\R^+)}$, where we used the fact that $N_V(x,E_V)=\frac{V_E-V_I }{x-V_I}\Pi(\R^+)$. At the time of death $\calT$, restart it (in a sense made precise through for instance the work of Meyer \cite{Meyer}) with distribution $\frac{N_V(X_{\calT-},dr)}{N_V(X_{\calT},E_V)}$, where $X_{t-}=\lim_{s\uparrow t}X_s$ stands for the left-limit, and, repeat the procedure. In other words, the neuronal model  $X$ behaves like  the classical Jacobi neuronal model but at some random times  performs downwards jumps (as we have, by definition, the support  of the kernel $N_V(x,dr)$ is $V_I<r<x$)   according to the distribution given above. In particular, the closer  $x$ gets to the inhibitory reversal potential $V_I$, the larger $N_V(x,E_V)$ is, that is the number of jumps becomes more frequent but the corresponding hyperpolarization is small, as the support of the  distribution of the amplitude of jumps is  $[0,x-V_I]$. Also, different choices of $\Pi$ allow different sizes of the jumps: the more mass $\Pi$ concentrates around zero, the smaller is the amplitude of jumps, if $\Pi$ admits large values with high probability, the voltage can be almost reset after the jump. We stress that the latter scenario could be the result of a rare event for the type of measures $\Pi$ considered in the following.
We consider the  case when $\Pi(dr)= e^{-\alpha r}dr, \alpha,r>0$. Then, easy computation yields that
$N_V(x,dr)=(V_E-V_I)\frac{(r-V_I)^{\alpha-1}}{(x-V_I)^{\alpha+1}}\mathbb{I}_{\{r<x\}} dr$ and thus $N_V(x,E_V)=\frac{V_E-V_I}{\alpha(x-V_I)}$.
 In this case, the probability that there is a jump of amplitude lower than  $y\in (0,x-V_I)$ is given by $\left(1-\frac{y}{x-V_I}\right)^{\alpha+1}$.

\begin{figure}[h]
\centering
\includegraphics[width=11 cm]{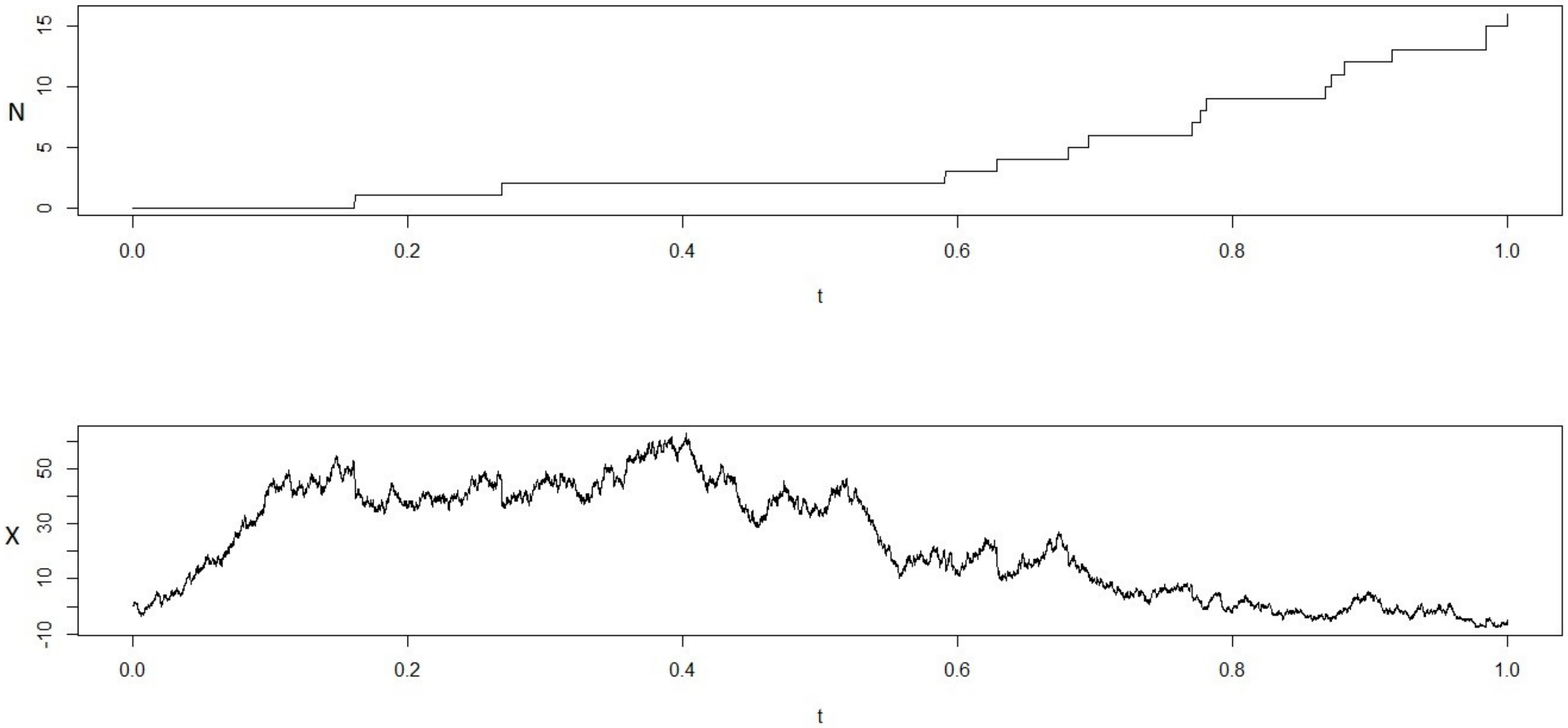}
\caption{A realization of a Jacobi process with jumps (bottom figure) and corresponding number of jumps (top figure). {The jump frequency increases if the depolarization $V_t$  approaches $V_I$ (in the figure $V_I=-10$, $V_E=100$), as it can be seen from the two plots for $t>0.6$, while the amplitude of the jumps decreases.}}\label{fig_traj}
\end{figure}

At first sight, one may be surprised  that the dynamics of the neuronal stochastic model is described  in terms of the generator compared to the usual path definition of diffusions as solution  to a stochastic differential equation. However, this is probably the most natural way when one is dealing with state dependent jumps processes. Indeed, the theory of stochastic differential equations for Markov processes with jumps is still incomplete regarding for instance the existence and uniqueness of a solution, and when available for state dependent jumps processes, it involves integral with respect to some Poisson random measures which makes its interpretation scarcely intuitive.

\begin{remark}
 For a better understanding of the dynamics and the interpretation of the involved parameters, we draw a parallel between the model \eqref{eq:gen_nontransf} and the equation of a conductance-based neuronal model. The evolution in time of the potential difference $\widetilde{V}=(\widetilde{V}_t)_{t\geq0}$ across the membrane of a neuron is given by
\begin{equation}
\label{conductance}
C \frac{d\widetilde{V}_t}{dt}=-g_L \widetilde{V}_t+I_{syn}+I
=-g_L \widetilde{V}_t-g_E(t)(\widetilde{V}_t-V_E)+g_I(t)(\widetilde{V}_t-V_I)+I,
\end{equation}
where $C>0$  is the membrane capacitance, $g_L>0$   is the conductance of the leak current, while $g_E(t)>0$ and $g_I(t)>0$ are the conductances of the excitatory and inhibitory components of the synaptic current $I_{syn}$, see e.g. \cite{rich2004}.
The current $I$ accounts for large or external inhibitory inputs that cannot be considered in the diffusion limit. Differently from classical external currents considered in the literature, that are constant or periodic functions, here $I$ has random components.
 The dynamics \eqref{conductance} with $I=0$ is analogous  to the deterministic version of the leaky integrate and fire model \eqref{SDE_LIF},
with $C/g_L$ playing the role of the membrane constant $\tau$.
 Moreover at time $t_k$, the time of the  arrival of the $k$-th incoming excitatory pulse, distributed according to a Poisson process with parameter $\nu_e$,  the conductance $g_E(t)$ increases by a factor of $C e$. Consequently the increase in the voltage is  $\Delta V=e(V_E-V_t)$, where $e$ is a dimensionless constant measuring the strength of the synapse (similarly for $g_I(t)$ and $\nu_i$). This increment corresponds to the one of the LIF model before taking the diffusion limit.
Analogously the current $I$ (the jump part in the proposed  model)  describes a possibly external (or strong) inhibitory current that at random times $\tau_k$, determined by a measure $\Pi$, decreases instantaneously the voltage of a quantity that depends on $N_V$ as described above, but in any case of a quantity smaller than $(V_{\tau_k}-V_I)$.
\end{remark}

The motivation for considering  the Jacobi process with state dependent jumps as  a model of neuron's activity are several folds.
On the one hand, the inhibition is well known to
be regulatory of neuronal excitability and has a role in information transmission.
The study, for state-dependent inputs, of the effect of inhibition on output indicators like signal to noise ratio,
effective diffusion coefficient of the spike count
and degree of coherence
 demonstrates that inhibitory input acts to decrease membrane potential fluctuations increasing spike regularity, see for example \cite{barta2019,barta2021,don_inhibition, sirovich}.
Moreover the state dependence of the jumps preserves the fundamental improvement with respect to the Ornstein-Uhlenbeck model that the changes in the potential depends on its actual value. In addition it allows the possible description of the sites in which the neuron receives the inputs giving the chance to relax the assumption that the cell is point-like.

We prove in Lemma \ref{lem:lk} that, under assumptions motivated by
realistic interpretation of the involved parameters, $X$ has only downward jumps.
This property suggests to apply, possibly,  the model to the probabilistic study of the effect of anti-epileptic drugs on a neuron whose firing activity is too intense, see \cite{epileptic}.

\section{Jacobi processes with jumps and their first passage time problems} \label{section3}
In this section, we start by providing a homeomorphism between the neuronal model $X$ defined in the previous section and generalized Jacobi processes with jumps that have been introduced  in \cite{CPSV}. Then, we proceed by characterizing the Laplace transform of the first passage time to a fixed level by these generalized Jacobi processes.
\subsection{Jacobi processes and neuronal models with jumps}
 Let us denote by $Y=(Y_t)_{t\geq0}$ the generalized Jacobi process with jumps defined in \cite{CPSV} as follows. It is a Feller process on $[0,1]$ whose infinitesimal generator is given, for a smooth  function $f$ on $[0,1]$, by
\begin{equation}
\label{eq:defJ}
\calJ_Y f(y) = \calJd_{\mu} f(y) +\int_0^\infty (f(e^{-r}y)-f(y)) \frac{\Pi(dr)}{y}
\end{equation}
where $\calJd_{\mu}$ is the classical Jacobi operator
\begin{equation*}
\calJd_{\mu} f(y)= \frac{\sigma^2}{2} y (1-y)f''(y)-\left(\lambda y-\mo\right)f'(y)
\end{equation*}
with $\sigma^2>0$, $\Pi$ is, as in \eqref{eq:gen_nontransf}, a finite, non-negative Radon measure on $\R_+$ with $\h = \int_0^\infty r\Pi(dr)< \infty$,  and, we have set, to simplify the notation,
\begin{equation}\label{eq:par}
\lam=\frac{1}{\tau} +\mu_e-\mu_i \textrm{ and } \mu=\mu_e-\frac{V_I}{\tau(V_E-V_I)},
\end{equation}
where these parameters were introduced in \eqref{input_par} and \eqref{eq:gen_nontransfJ}. Throughout, we impose  the following assumption
that guarantees that $V_I$ is an entrance boundary
 \begin{equation}\label{ass:1}
\mu_e  > \h+\frac{\sigma^2}{2}+\frac{V_I}{\tau (V_E-V_I)} \textrm{ or, equivalently, } \lambda > \mo  > \h+\frac{\sigma^2}{2},
\end{equation}
 since $\mu_i<0$. Note that in \cite{CPSV}, $\sigma^2=2$, and we shall explain in Lemma \ref{lem:sigma} below how to relate our generator $\calJ_Y$ to the one of \cite{CPSV}. In what follows, we recall some basic results from \cite{CPSV} that will be useful for our analysis by adapting them, in an obvious way, for any $\sigma^2>0$. For instance, the condition \eqref{ass:1} is the appropriate modification of  the standing assumption in \cite{CPSV}.  
Next, let us write, for $u\geq0$,
\begin{equation} \label{def:bern}
\phi(u) =  u + \frac{2}{\sigma^2}\left(\mo-\h-\frac{\sigma^2}{2} + \int_0^\infty (1-e^{-ur})\overline{\Pi}(r)dr\right)
\end{equation}
where $\overline{\Pi}(r)=\int_{r}^{\infty}\Pi(du)$. We observe that the condition \eqref{ass:1}, i.e.~$\mo  > \h+\frac{\sigma^2}{2}$, is also equivalent to
 \begin{equation}\label{eq:cond_phi}
 \phi(0)=\frac{2}{\sigma^2}\left(\mo-\h-\frac{\sigma^2}{2}\right) >0.
  \end{equation}
  Under this condition, it is not difficult to check that $\phi$ is a Bernstein function, i.e.~$\phi:[0,\infty) \to [0,\infty)$ is infinitely differentiable on $\R_+$ and $(-1)^{n+1}\frac{d^n }{du^n}\phi(u)\geq 0$, for all $n = 1,2,\ldots$ and $u \geq 0$, see Schilling et al.~\cite{SchillingSongVondracek10} for a thorough exposition on Bernstein functions and subordinators.
We denote throughout by $\mathbf{B}_J$ the subset of Bernstein functions of the form \eqref{def:bern} which satisfies the condition $\phi(0)>0$.

We also observe that $\calJ_Y$ (resp.~$\phi$) is uniquely determined by $\sigma^2,\Pi,\mo$ and $\lam$ (resp.~$\sigma^2,\Pi,\mo$) so that, for a fixed $\lam$, there is a one-to-one correspondence between $\phi$ and $\calJ_Y$.

Next,  we set, $W_{\phi}(1)=1$ and for any $n\geq1$,
\begin{equation}
\label{eq:product-Wphi}
W_{\phi}(n+1)=\prod_{k=1}^n \phi(k)
\end{equation}
 Note that $W_{\phi}$ is solution to the recurrence equation $
W_{\phi}(n+1)=\phi(n)W_{\phi}(n),$
with $W_{\phi}(0)=1$, and we refer to Patie and Savov~\cite{Patie-Savov-Bern} for a thorough account on this set of functions that generalizes the gamma function, which appears as a special case when $\phi(n)=n$. Then,  it is shown, in \cite[Theorem 2.1]{CPSV}, that there exists  an absolutely continuous probability measure whose support is $[0,1]$, with a continuous density denoted by $\bpsi$ that is positive on $(0,1)$. Being of compact support, its law is  moment determinate, and, more specifically, one has, for any $n \in \N$,
\begin{equation}
\label{eq:mom-bpn}
\int_{0}^{1}y^n\bpsi(y)dy = \frac{W_{\phi}(n+1)\Gamma(\frac{2\lam}{\sigma^2})}{\Gamma(\frac{2\lam}{\sigma^2}+n)}.
\end{equation}
 Note that, in particular, using \eqref{def:bern}, one gets the following expression for the first moment of $\beta$
\begin{equation}
\label{eq:mom-bpn_1}
\int_{0}^{1}y\bpsi(y)dy = \sigma^2\frac{\phi(1)}{2\lam}= \sigma^2\frac{(\frac{2}{\sigma^2}\mo-\frac{2}{\sigma^2}\h) + \frac{2}{\sigma^2}\int_0^\infty (1-e^{-r})\overline{\Pi}(r)dr}{2\lam}=\frac{\mo - \int_0^\infty e^{-r}\overline{\Pi}(r)dr}{\lam}.
\end{equation}
We also point out that when  $\phi(u) = u$, $\bpsi$ boils down to  the Beta distribution, which is easily identified from the expression of its moments as in this case $W_\phi(n+1)=n!$. Other examples will be provided in Section \ref{sec:ex}.
$\beta$ turns out to be the stationary measure of the  Feller semigroup $\J$, that is for all $f \in C([0,1])$, the Banach space of continuous functions on $[0,1]$ equipped with the sup-norm $||\cdot||_\infty$, and $t \geq 0$,
\begin{equation} \label{eq:inva}
\bpsi [ \J_t f ] = \bpsi[f] =\int_0^1 f(y)\bpsi(dy)
\end{equation}
where the last equality serves as a definition for the notation $\bpsi[f]$.
The extension of $\calJ_Y$ to an operator on $\Leb^2(\bpsi)$, still denoted by $\calJ_Y$, is the infinitesimal generator, having $\Poly$, the algebra of polynomials, as a core, of an ergodic Markov semigroup $\J = (\J_t)_{t \geq 0}$ on $\Leb^2(\bpsi)$ whose unique invariant measure is $\bpsi$.

It is then classical, see either Bakry et al.~\cite{Bakry_Book} or Da Prato \cite{da-prato:2006}, that given a {Markov semigroup on $C([0,1])$} with invariant probability measure $\bpsi$ one may extend it to a Markov semigroup on $\Leb^2(\bpsi)$, the weighted Hilbert space being defined as
\begin{equation*}
\Leb^2(\bpsi) = \left\lbrace f: [0,1] \to \R \textrm{ measurable with } \bpsi[f^2] <\infty \right\rbrace.
\end{equation*}
Such a semigroup is said to be ergodic if, for every $f \in \Leb^2(\bpsi)$, $\lim_{T \to \infty} \frac{1}{T} \int_0^T \J_t f dt = \bpsi[f]$ in the $\Leb^2(\bpsi)$-norm.

\begin{proposition}\label{prop:intD}
Let $X=(X_t)_{t\geq0}$ where, for any $t\geq0$, $X_t=g(Y_t)$ with $g(x)=(V_E-V_I)x+V_I$. Then $X$ is a Feller process on $E_V=[V_I,V_E]$ which admits the measure $(V_E-V_I)^{-1}\beta\left(\frac{x-V_I}{V_E-V_I}\right)dx$ as the unique stationary measure. Its infinitesimal generator is the closure of $(\calJ_X, \Poly)$, where $\calJ_X$ is defined in \eqref{eq:defJ} and $\Poly$ is a core. Moreover, we have, on $\Poly$,
\begin{equation}\label{eq:intD}
 \calJ_X G f = G\calJ_Y f
\end{equation}
where $G f(x)=f \circ g(x)$ is a homeomorphism from $[0,1]$ onto $[V_I,V_E]$.
\end{proposition}
\begin{proof}
Since $g$ is a homeomorphism from $[0,1]$ onto $[V_I,V_E]$ with inverse function $h(x)=\frac{x-V_I}{V_E-V_I}$ and, from \cite[Lemma 3.10 and its proof]{CPSV}, $Y$ is a Feller process on $[0,1]$, we deduce that $X$ is also a Feller process on $E_V$.
Next, using \eqref{eq:gen_nontransf} and the notation \eqref{eq:par},   simple algebra  yields, for any  $f \in \Poly$,  
\begin{eqnarray}
G^{-1}\calJ_X Gf(y) &=& \sigma^2 (V_E-V_I)^2y(1-y)f''(y)\frac{1}{(V_E-V_I)^2} \nonumber \\
&-&\left(\left(\frac{1}{\tau} +\mo_e-\mo_i \right)(V_E-V_I)y-\mo_e(V_E-V_I)+\frac{V_I}{\tau} \right)f'(y)\frac{1}{(V_E-V_I)} \nonumber \\
&+&\int_0^\infty \left(f\left(e^{-r}y\right)-f\left(y\right)\right) \frac{\Pi(dr)}{y} \nonumber \\
 &=& \sigma^2 y(1-y)f''(y)
-\left(\lam y-\mo \right)f'(y) +\int_0^\infty \left(f\left(e^{-r}y\right)-f\left(y\right)\right) \frac{\Pi(dr)}{y}\nonumber \\
&=&\calJ_Y f(y)\nonumber
\end{eqnarray}
which completes the proof of the intertwining relation. Since $\Poly$ is core for $\calJ_Y$, see \cite[Theorem 2.1]{CPSV} we deduce, by the homeomorphism $G$, that $\Poly$ is also a core for $\calJ_X$. Next, by taking the inverse of $G$, from the left and from the right, in the relation \eqref{eq:intD}, one gets that $ G^{-1}\calJ_X f = \calJ_Y G^{-1} f$. Since from \cite{CPSV}, we have that $\beta$ is the unique stationary measure for $Y $, that is the unique measure $\beta$ such that $\beta \calJ_Y  f=0, f\in \mathcal{D}_Y$, we deduce  that the measure on $[V_I,V_E]$, defined by  $\beta_G=\beta G^{-1}$, is the unique one such that $\beta_G\calJ_X G f=0$, which completes the proof.
\end{proof}
\begin{remark}\label{rem_markov}
We point out that the requirement that  the support of the measure of the Lévy kernel  is the negative half-line (or, equivalently, the process has only negative jumps) comes from Proposition \ref{prop:intD}. Indeed, it ensures that the corresponding integro-differential operator satisfies the maximum principle, and, hence it is the generator of a Markov semigroup. The interested reader can consult \cite{CPSV} for a detailed discussion on this technical aspect.
\end{remark}

\subsection{Laplace transform of first passage times}
Let   us write
\begin{equation}\label{eq:def_T}
  T_a = \inf\{t>0;\: Y_t\geq a\}
\end{equation}
for the first passage time to the level $0<a<1$ of the generalized Jacobi process $Y$. Note that, from  Proposition \ref{prop:intD}, one gets, when $Y$ is issued from $y\in (0,1)$, the identity in distribution, with the obvious notation,
$T_{y\rightarrow a}(Y)\stackrel{d}{=}T_{g(y)\rightarrow g(a)}(X)$.
To characterize the Laplace transform of $T_a$, we introduce the mapping
\begin{equation}\label{eq:HypPhi}
  {}_2F^{}_1\left(a,b,\phi;y\right) = \sum_{n=0}^{\infty}\frac{(a)_n (b)_n}{n!}\frac{y^n}{W_{\phi}(n+1)}
\end{equation}
with $(a)_n = \frac{\Gamma(a+n)}{\Gamma(a)}, n \in \N, a \in \mathbb{C}$. Note that when $\Pi\equiv 0$, we have, from \eqref{def:bern}, $W_{\phi}(n+1)= (\phi(0)+1)_n$,
and, thus, in this case \[ {}_2F_1\left(a,b;\phi; y\right)= {}_2F_1\left(a,b;\phi(0)+1;y\right)\] which is the Gauss hypergeometric function, explaining the notation. There are several representations of this function which provides an analytical continuation to the entire complex plane cut along $[1,\infty]$ with $\lim_{x\uparrow 1}{}_2F^{}_1\left(a,b,1;z\right)=\frac{\Gamma (1-a-b)}{\Gamma (1-a)\Gamma (1-b)}, \Re(a+b)<1$, see \cite[Chap.~9]{Lebedev}.  We are now ready to state the following.

\begin{theorem}\label{thm:fpt}
Let $\phi \in \mathbf{B}_J$. Then, for any $a,b\in \mathbb{C}$, the mapping $z\mapsto {}_2F_1\left(a,b;\phi;z\right)$ defines an analytic function on the unit disc. Moreover,  for any $0<y<a<1$ and $q>0$, we have
\begin{equation}
\label{Lapl_bo3}
\mathbb{E}_{y}\left[e^{-q T_a}\right]=\frac{{}_2F_1\left(\kappa(q),\theta(q);\phi;y\right)}
{{}_2F_1\left(\kappa(q),\theta(q);\phi;a\right)},
\end{equation}
where $\kappa(q)$ and $\theta(q)$ are solution to the system
\begin{equation}
\kappa(q)\theta(q)=\frac{2 q}{ \sigma^2} \textrm{ and } \kappa(q)+\theta(q)+1=\frac{2 \lambda}{ \sigma^2}.
\label{k_theta}
\end{equation}
\end{theorem}
\subsection{Proof of Theorem \ref{thm:fpt}}The proof is split into several intermediate results. We start with the following result that shows that the dynamics of $Y$ has discontinuities which are due to negative jumps only.
\begin{lemma}\label{lem:lk}
We have, for all $y\in [0,1], t\geq0$  and $\tilde{f}$ a positive borelian function on $[0,1]\times [0,1]$,
\begin{equation*}
  \E_y\left[\sum_{s\leq t}\tilde{f}(Y_{s-},Y_s)\mathbb{I}_{\{Y_{s-}\neq Y_s\}}\right]=\E_y\left[\int_0^tds \int_0^1\tilde{f}(Y_{s-},r)\mathbb{I}_{\{r\leq Y_{s-}\}}\widetilde{\Pi}(Y_{s-},dr)\right]
\end{equation*}
where $\widetilde{\Pi}(y,.)$ is the image measure of $\frac{\Pi(.)}{y}$ by the mapping $r\mapsto - \ln(r/y)$. Consequently,
 for all $y\in [0,1]$, $\Prob_y(Y_{t-}\geq Y_t \textrm{ for all } t\geq0 )=1$, i.e.~$Y$ has only downward jumps.
 \end{lemma}
\begin{proof}
First, by \cite[Lemma 3.1]{CPSV}, we know that $Y$ is a Feller process, and hence, from \cite{Ben}, we have that $Y$ admits a L\'evy kernel, say $N$, that we now characterize.
To this end, one observes, from \eqref{eq:defJ},  that for a smooth function $f$ that vanishes in the neighborhood of $y \in [0,1]$, we have
\begin{equation}
\calJ f(y) = \int_0^\infty f(e^{-s}y) \frac{\Pi(ds)}{y}= \int_0^1 f(s)\mathbb{I}_{\{r\leq y\}}\widetilde{\Pi}(y,dr)
\end{equation}
 where $\widetilde{\Pi}(y,.)$ is the measure defined in the claim. Hence, the L\'evy kernel $N(y,dr)=\mathbb{I}_{\{r\leq y\}}\widetilde{\Pi}(y,dr)$,  see e.g.~\cite{Meyer}. The first claim follows from the definition of the L\'evy kernel whereas the second one is deduced from the first one by choosing the function $\tilde{f}(y,r)=\mathbb{I}_{\{y\leq r\}}$.
\end{proof}
We proceed with the following.
\begin{lemma}
Let us write $F_{q}(y)={}_2F_1\left(\kappa(q),\theta(q);1;y\right), q>0,$ then $F_q$ is positive increasing on $(0,1)$ and we have
\begin{equation*}
\calJd F_{q}(y)= q F_{q}(y), \quad y\in [0,1],
\end{equation*}
where, simplifying the notation, $\calJd=\calJd_{\frac{\sigma^2}{2}}$, that is $\calJd f(y)= \frac{\sigma^2}{2} y (1-y)f''(y)-\left(\lam y-\frac{\sigma^2}{2}\right)f'(y)$. Consequently, for any $t,q\geq 0$ and $y\in [0,1)$,
\begin{equation} \label{eq:iFq}
e^{-qt}\Q_t F_{q}(y)= F_{q}(y)
\end{equation}
that is, $F_{q}$ is a $q$-invariant function for $\Q=(\Q_t)_{t\geq0}$ the semigroup  associated to $\calJd $.
\end{lemma}
\begin{proof}  The first part is classical, see Appendix \ref{appendix:class_jac}. Note that $\kappa(q)+\theta(q)=1-2\lambda/\sigma^2<1$ which ensures that $\lim_{y\uparrow 1}{}_2F^{}_1\left(\kappa(q),\theta(q),1;y\right)$ exists.   Next, using the fact that, in addition, the mapping $x\mapsto F_{q}(y) $ is twice continuously differentiable on $[0,1]$, one can apply It\^o's formula to get
\begin{eqnarray*}
  e^{-qt}F_{q}(Y_t) &=& F_{q}(y) +\int_{0}^{t} \calJd F_{q}(Y_s)-qF_{q}(Y_s)ds+ \sqrt{2}\sigma \int_{0}^{t} \sqrt{ Y_s (1-Y_s)}F_{q}'(Y_s)dB_s \\
  &=& F_{q}(y) + \sqrt{2}\sigma \int_{0}^{t} \sqrt{ Y_s (1-Y_s)}F_{q}'(Y_s)dB_s.
\end{eqnarray*}
Since the last term has a squared integrable integrant, it defines a martingale. Then, taking the expectation on both sides of the previous identity yields the second claim.
\end{proof}
Let us now  denote by $\varrho=(\varrho_t)_{t\geq0}$  a subordinator that is  a positive valued stochastic process with stationary and independent increments,  and recall that its law is uniquely determined by a Bernstein function $\phi$. More specifically, one has, for any $t,u\geq0$,
\begin{equation} \label{eq:LambdaF}
\E[e^{-u\varrho_t}]=e^{-\phi(u)t}.
\end{equation}
Next, for each subordinator $\varrho$ associated to $\phi \in \mathbf{B}_J$, we define the random variable
 \[ I_{\phi} = \int_{0}^{\infty}e^{-\varrho_t}dt \]
 which is the so-called exponential functional of the subordinator $\varrho$. We point out that this random variable has been studied intensively over the last two decades see e.g.~\cite{Patie-Savov-Bern} and the references therein.
\begin{lemma}
Let $\phi \in \mathbf{B}_J$ and write  $F_{_{a,b}}(y)={}_2F_1\left(a,b;1;y\right)$. Then, we have, for any
 $y\in [0,1]$,
\begin{equation} \label{eq:LambdaF}
\Lambda_{\phi} F_{_{a,b}}(y)=  {}_2F_1\left(a,b;\phi;y\right)
\end{equation}
where $\Lambda_{\phi}: C([0,1])\mapsto  C([0,1])$ is the Markov multiplicative operator associated to the random variable $I_{\phi}$, that is
\begin{equation}
\Lambda_{\phi} f(y)= \E \left[f(yI_{\phi})\right].
\end{equation}
Moreover, $z\mapsto {}_2F_1\left(a,b;\phi;z\right)$ defines a function which is analytic on the unit disc.
\end{lemma}

\begin{proof}
  First, we recall, from e.g.~\cite[Lemma 3.3]{CPSV},  that $\Lambda_{\phi}$ is a Markov bounded operator from $C([0,1])$ into itself, and,  with  $p_n(y) = y^n, n \in \N$,
\begin{eqnarray}\label{eq:momIp}
\qquad \Lambda_{\phi}p_n(y) = \frac{n!}{W_\phi(n+1)}p_n(y).
\end{eqnarray}
Then, an application of Tonelli Theorem and \eqref{eq:momIp} yield, for any $0\leq y\leq 1$,
\begin{equation*}
\Lambda_{\phi} F_{a,b}(y)= \E \left[F_{a,b}(yI_{\phi})\right] =  \sum_{n=0}^{\infty} \frac{(a)_n (b)_n}{n!}\frac{\Lambda_{\phi}p_n(y)}{n!}= {}_2F_1\left(a,b;\phi;y\right).
\end{equation*}
Moreover, as $W_{\phi}(n+2)=\phi(n+1)W_{\phi}(n+1)$ and $\lim_{n\to\infty}\frac{\phi(n+1)}{n+1}=1$, we easily get that the power series ${}_2F_1\left(a,b;\phi;.\right)$ defines an analytic function on the unit disc.
\end{proof}

\begin{lemma} \label{lem:sigma}
Writing $F^{(\phi)}_{q}(y)={}_2F_1\left(\kappa(q),\theta(q);\phi;y\right)$, we have,
 for any $t,q\geq 0$,
\begin{equation*}
e^{-qt}\J_t F^{(\phi)}_{q}(y)= F^{(\phi)}_{q}(y)
\end{equation*}
that is $F^{(\phi)}_{q}$ is a $q$-invariant function for $\J$, which is positive and increasing on $[0,1]$.
\end{lemma}
\begin{proof}
First, let us denote by $(\widetilde{\J}_{t})_{t\geq0}$ the semigroup associated to the non-local Jacobi generator given, for any $\sigma^2>0$ and $y\in (0,1)$, by
 \begin{eqnarray}
\label{eq:gen_transf}
\widetilde{\calJ}^Y f(y) &=&  y(1-y)f''(y)
-\left(\frac{\lam}{\sigma^2} y-\frac{\mo}{\sigma^2} \right)f'(y) +\int_0^\infty \left(f\left(e^{-r}y\right)-f\left(y\right)\right) \frac{\Pi(dr)}{y\sigma^2}.
\end{eqnarray}
Then, observes that, for any $c>0$,
\begin{eqnarray}
\label{eq:gen_transf}
\calJ^Y f(y) &=& \lim_{t\to 0}\frac{\J_tf(y)-f(y)}{t}= {c}\lim_{t\to 0}\frac{\widetilde{\J}_{ct}f(y)-f(y)}{{ct}} =c\widetilde{\calJ}^Y f(y)
\end{eqnarray}
and note that the same relationship holds between $(\Q_t)_{t\geq0}$ and $(\widetilde{\Q}_t)_{t\geq0}$ the semigroups of the classical Jacobi processes with generator  $\calJd f(y)=  \frac{\sigma^2}{2} y (1-y)f''(y)-\left(\lam y-\frac{\sigma^2}{2}\right)f'(y)$ and $\widetilde{\calJd} f(y)=  y (1-y)f''(y)-\left(2\lam \sigma^{-2}y-1\right)f'(y)$ respectively. Now,  we recall from \cite[Proposition 3.3]{CPSV}, taking with the notation thereout $\epsilon=d_{\phi}$ and $r_1=1$, that the following intertwining relation
 \begin{eqnarray}\label{eq:inter-right-d0}\label{eq:inter-left-m}\label{eq:inter-left-varphi}
\widetilde{\J}_t \Lambda_{\phi} = \Lambda_{\phi} \widetilde{\Q}_t
\end{eqnarray}
holds on the weighted Hilbert space $\Leb^2(\gab[\lambda_1])$, where $\lambda_1=2\lambda \sigma^{-2}>1$, by the assumption \eqref{ass:1}, and, $\gab[\lambda_1](dy)=(\lambda_1-1)(1-y)^{\lambda_1-2}dy,  y\in(0,1)$. Hence, for any $t\geq 0$, on $\Leb^2(\gab[\lambda_1])$,
 \begin{eqnarray}\label{eq:inter-right-d0}\label{eq:inter-left-m}\label{eq:inter-left-varphi}
\J_t \Lambda_{\phi}=\widetilde{\J}_{\sigma^2/2t}\Lambda_{\phi} = \Lambda_{\phi} \widetilde{\Q}_{\sigma^2/2t}=\Lambda_{\phi} \Q_{t}.
\end{eqnarray}
 Thus, using successively that $F_{q} \in \Leb^2(\gab[\lambda_1])$,  \eqref{eq:LambdaF}, \eqref{eq:inter-right-d0} and \eqref{eq:iFq}, one gets that
 \begin{eqnarray*}
e^{-qt}\J_t F^{(\phi)}_{q}(y)= e^{-qt}\J_t^\phi \Lambda_{\phi}F_{q}(y) =  e^{-qt} \Lambda_{\phi} \Q_tF_{q}(y)=  \Lambda_{\phi} F_{q}(y) = F^{(\phi)}_{q}(y)
 \end{eqnarray*}
 which proves the first claim. Next, since $\Lambda_{\phi}$ is clearly a Markov operator, i.e.~$\Lambda_{\phi} f\geq 0$ for any $f\geq0$ and $\Lambda_{\phi}p_0(y)=1$, we get that ${}_2F_1\left(a,b;\phi;.\right)\geq0$ on $[0,1]$. Finally, $F^{(\phi)}_{q}$ being a power series with non-negative coefficients, we deduce the monotonicity property.
\end{proof}
\subsubsection*{End of the proof of Theorem \ref{thm:fpt}}
 First, one invokes the previous lemma and Dynkin's theorem to the bounded stopping time $T^t_a=T_a\wedge t$, to get, for any $t,q>0$ and $0<y<a<1$,
 \begin{eqnarray*}
\E_y\left[ e^{-qT^t_a} F^{(\phi)}_{q}(Y_{T^t_a})\right]=  F^{(\phi)}_{q}(y).
 \end{eqnarray*}
Then, letting $t \to \infty$, using the fact that $F^{(\phi)}_{q}$ is increasing on $[0,1]$ and by absence of positive jumps, see Lemma \eqref{lem:lk}, $\mathbb{P}_y(X_{T_a}=a)=1$, combined with a dominated convergence argument yield
\begin{eqnarray*}
\E_y\left[ e^{-qT_a}\mathbb{I}_{\{T_a<\infty\}} \right]=  \frac{F^{(\phi)}_{q}(y)}{F^{(\phi)}_{q}({a})}.
 \end{eqnarray*}
 Next, observe that, if $\overline{\lam}=\frac{\lam}{\sigma^2}-\frac12\geq0$ (resp.~$<0$) then, by Taylor's expansion,  one gets that $\lim_{q\to 0}\sigma^2\overline{\lam}\frac{\theta(q)}{q}=1$ (resp.~$\lim_{q\to 0}\theta(q)=2\overline{\lam}$), and thus $\lim_{q\to 0}\kappa(q)=2\overline{\lam}$ (resp.~$=0$). It is not difficult to check that in both cases, one has for all $y \in [0,1)$, $\lim_{q\to 0} F^{(\phi)}_{q}(y)=1$ and hence  $\mathbb{P}_y\left(T_a<\infty \right)=1$, which completes the proof of Theorem \ref{thm:fpt}. \\

\subsection{Mean of the first passage times}
We proceed by deriving the expression  of the first moment of  the first passage times of our  family of Markov processes whose proof is split into several intermediate results.

\begin{theorem} \label{thm:MFPT}
 Let $\phi \in \mathbf{B}_J$. Then, for any $0<y<a<1$,
\begin{equation}
\label{ET_nonloc_jacobi}
\mathbb E_y[T_a]= \frac{2}{\sigma^2} \sum_{n=0}^{\infty}\frac{(2\lam/\sigma^2)_n}{n+1}\frac{a^{n+1}-y^{n+1}}{W_{\phi}(n+2)}.
\end{equation}
\end{theorem}
\begin{remark}
We note, from \eqref{ET_nonloc_jacobi}, that when $\Pi=0$, we recover the expression of $\mathbb E_y[T_a]$ for the classical Jacobi process given in \eqref{ET_classical}. Indeed, in this case,
$$W_\phi(n+2)=\prod_{k=1}^{n+1}\phi(k)=\left(\frac{2\mu}{\sigma^2}\right)_{n+1}=\frac{2\mu}{\sigma^2}\left(\frac{2\mu}{\sigma^2}+1\right)_{n}.$$
\end{remark}
As a by-product of this Theorem, we state and prove the following comparison result between the first moment of the first passage times of our class of Jacobi processes with jumps.
\begin{corollary} \label{cor:comp}
Let $ \phi, \phi_1 \in \mathbf{B}_J$ be such that $\phi \leq \phi_1$ on $\R_+$.  Then, we have,  for any $0<y<a<1$, with the obvious notation, 
\begin{equation} \label{eq:CET}
\mathbb E_y[T^{\phi_1}_a] \leq \mathbb E_y[T^{\phi}_a]
\end{equation}
The conditions hold, for instance, when  $ \phi, \phi_1 \in \mathbf{B}_J$  and $\overline{\Pi}\geq \overline{\Pi}_1$ on $\R_+$, or when $\phi \in \mathbf{B}_J$ and $\phi_1(u)=\phi(u)+\frac{2}{\sigma^2}\int_0^\infty (1-e^{-ur})\overline{\Pi}_1(r)dr, u\geq0,$ where $\overline{\Pi}_1$ satisfies the same conditions than $\overline{\Pi}$ in \eqref{eq:defJ}.  Note that, with the notation of \eqref{eq:cond_phi},  in the former instance, we have $\mu=\frac{\sigma^2}{2}(\phi(0)+1)+\h=\frac{\sigma^2}{2}(\phi_1(0)+1)+\h_1=\mu_1$ whereas, in the latter case, $\mu_1=\mu+\h_1$.   
\end{corollary}
\begin{remark}
 It is interesting to note that, in the second example, i.e.~when $\mu_1=\mu+\h_1$, although the intensity of jumps is larger for the Jacobi process associated to $\phi_1$, the first moment of its first passage times above the starting point is smaller, meaning that adding the mean of the jump measure to the drift term compensates the presence of additional downwards jumps.   
\end{remark}
\begin{proof}
 We first observe, that  with $\phi,\phi_1 \in \mathbf{B}_J$ such that  $\phi\leq \phi_1$ on $\R_+$, we have, for all $n \geq0, W_{\phi}(n+1) \leq W_{\phi_1}(n+1)$, the inequality  between the first moments of the first passage times is deduced easily from the identity \eqref{ET_nonloc_jacobi} as the other parameters in the two expressions are identical. Finally, on the one hand, since   $\h=\int_0^\infty\overline{\Pi}(r)dr$ we get, from \eqref{def:bern},  that 
\begin{equation} \label{eq:defphi1} 
\phi(u) =  u + \frac{2}{\sigma^2}\left(\mo-\frac{\sigma^2}{2} - \int_0^\infty e^{-ur}\overline{\Pi}(r)dr\right)
\end{equation}
which provides the first inequality between the Bernstein functions. On the other hand,  
 writing $\h_1=\int_0^\infty \overline{\Pi}_1(r)dr$, we get that  \[ \phi_1(u) =  u + \frac{2}{\sigma^2}\left(\mo+\h_1-(\h+\h_1)-\frac{\sigma^2}{2} + \int_0^\infty (1-e^{-ur})(\overline{\Pi}+\overline{\Pi}_1)(r)dr\right)  \in \mathbf{B}_J \]
 Then, plainly  $\phi\leq \phi_1$ on $\R_+$, and the last remarks follow readily. 
\end{proof}
We now turn to the proof of  Theorem \ref{thm:MFPT} which relies on taking the derivative of the Laplace transform \eqref{Lapl_bo3} which is given in the following lemma. This extends  the result of \cite{ancarani2009} on the derivative of the Gauss hypergeometric function ${}_2F_1$.

\begin{lemma}\label{lemma_derivative}
 Let $\phi \in \mathbf{B}_J$. Then, for any $|z|<1$,
\begin{eqnarray}
\frac{\partial}{\partial a}\:{}_2F_1(a,b,\phi,z)_{|_{a=0}}&=&
b  \sum_{n=0}^{\infty}\frac{(b+1)_n}{n+1}\frac{z^{n+1}}{W_{\phi}(n+2)} \label{derivative2f1azero}\\
\frac{\partial}{\partial b}\:{}_2F_1(a,b,\phi,z)_{|_{a=0}}&=&0.
\end{eqnarray}
\end{lemma}
\begin{proof}
Using that $\frac{\partial}{\partial a}\:(a)_n=(a)_n[\Psi(a+n)-\Psi(a)]$, where $\Psi$ is the Digamma function, we get that
\begin{equation}
\frac{\partial }{\partial a}\:{}_2F_1(a,b,\phi,z)=
\sum_{n=0}^\infty \frac{(a)_n(b)_n}{n!}\left(\Psi(a+n)-\Psi(a)\right)\frac{z^n}{W_{\phi}(n+1)}.
\end{equation}
From  \cite[Formulas 6.3.5 and 6.3.6]{abr_steg}, one has
\begin{equation}
\Psi(a+n)-\Psi(a)=\sum_{k=0}^{n-1}\frac{1}{k+a}.
\end{equation}
Moreover, observing that
\begin{equation}
\frac{1}{k+a}=\frac{1}{a}\frac{(a)_k}{(a+1)k}
\end{equation}
one gets
\begin{eqnarray}
\frac{\partial }{\partial a}\:{}_2F_1(a,b,\phi,z)&=&
\frac{1}{a}\sum_{n=0}^\infty\sum_{k=0}^{n} \frac{(a)_{n+1}(a)_k(b)_{n+1}}{(a+1)_k (n+1)!}\frac{z^{n+1}}{W_{\phi}(n+2)} \nonumber 	\\
&=&\frac{1}{a}\sum_{n=0}^\infty\sum_{k=0}^{\infty} \frac{(a)_{n+k+1}(a)_k(b)_{n+k+1}}{(a+1)_k (n+k+1)!}\frac{z^{n+k+1}}{W_{\phi}(n+k+2)} \nonumber \\
&=&bz\sum_{n=0}^\infty\sum_{k=0}^{\infty} \frac{(a+1)_{n+k}(a)_k(b+1)_{n+k}}{(a+1)_k (n+k+1)!}\frac{z^{n+k}}{W_{\phi}(n+k+2)}\label{derivative2f1}
\end{eqnarray}
where in the last equality we have used that $(a)_{n+k+1}=a (a+1)_{n+k}$.
From  \eqref{derivative2f1} with $a=0$, noting that we have non-zero terms only for $k=0$, the identity  \eqref{derivative2f1azero} follows.
The expression of $\frac{\partial}{\partial b} \ {}_2F_1(a,b,\phi,z)$ can be obtained directly by interchanging $a$ with $b$ in \eqref{derivative2f1}, i.e
\begin{equation}
\frac{\partial}{\partial b}\:{}_2F_1(a,b,\phi,z)=az\sum_{n=0}^\infty\sum_{k=0}^{\infty} \frac{(b+1)_{n+k}(b)_k(a+1)_{n+k}}{(b+1)_k (n+k+1)!}\frac{z^{n+k}}{W_{\phi}(n+k+2)}
\end{equation}
that is always equal to zero for $a=0$.
\end{proof}

\subsubsection*{End of the proof of Theorem \ref{thm:MFPT}}
 For $\phi \in \mathbf{B}_J$, that is $\lam > \mo  > \h+\frac{\sigma^2}{2}$, we have, writing $\Phi_y(q):=\mathbb{E}_y[e^{-q T_a}\mathbb{I}_{\{T_a<\infty\}}]$,
\begin{eqnarray}
\mathbb E_y[T_a]&=&-\frac{\partial \Phi_y(q)}{\partial q}{\Big|_{q=0}}=-\Phi_y(0)  \frac{\partial \ln \Phi_y(q)}{\partial q}{\Big|_{q=0}}=-\frac{\partial \ln \Phi_y(q)}{\partial q}{\Big|_{q=0}}\nonumber \\
&=&\frac{\partial}{\partial q}\left({}_2F_1\left(\kappa(q),\theta(q);\phi;a\right)-{}_2F_1\left(\kappa(q),\theta(q);\phi;y\right)\right){|_{q=0}}.
\label{ETa}
\end{eqnarray}
We have
\begin{eqnarray}
\frac{\partial}{\partial q}{}_2F_1\left(\kappa(q),\theta(q);\phi;y\right){|_{q=0}}&=&
\frac{\partial}{\partial \kappa(q)}{}_2F_1\left(\kappa(q),\theta(q);\phi;y\right){|_{q=0}}\frac{\partial \kappa(q)}{\partial q} \Big|_{q=0} \nonumber \\
&+&\frac{\partial}{\partial \theta(q)}{}_2F_1\left(\kappa(q),\theta(q);\phi;y\right){|_{q=0}}\frac{\partial \theta(q)}{\partial q}{\Big|_{q=0}}.
\label{chian_rule}
\end{eqnarray}
and, with $\bar \lam:= \frac{\lam}{\sigma^2}-\frac12 \geq 0$, it is easy to check, from the system \eqref{k_theta},  that
\begin{equation}
\label{res_1}
\frac{\partial \kappa(q)}{\partial q} {|_{q=0}}=\frac{1}{\bar \lam \sigma^2}, \qquad \kappa(0)=0,
\qquad \theta(0)=2\bar \lam.
\end{equation}
Moreover, one gets, by Lemma \ref{lemma_derivative}, that
\begin{eqnarray}
\frac{\partial}{\partial \kappa(q)}{}_2F_1\left(\kappa(q),\theta(q);\phi;y\right){|_{q=0}}&=&
\frac{\partial}{\partial \kappa(q)}{}_2F_1\left(\kappa(q),\theta(q);\phi;y\right)|_{\kappa(q)=0} \nonumber \\
&=&2\overline{\lam}  \sum_{n=0}^{\infty}\frac{(1)_n(2\bar\lam+1)_n}{(2)_n}\frac{y^{n+1}}{W_{\phi}(n+2)}, \label{res_2}
 \end{eqnarray}
 and
 \begin{eqnarray}
\frac{\partial}{\partial \theta(q)}{}_2F_1\left(\kappa(q),\theta(q);\phi;y\right)|_{\kappa(q)=0}&=&0. \label{res_3}
\end{eqnarray}
Finally, combining  \eqref{ETa} and \eqref{chian_rule}-\eqref{res_3},
we obtain the expression of $\mathbb{E}_y[T_a]$, which completes the proof of the Theorem. 

\section{Firing activity of the Jacobi process with jumps}

Let $X$ be the Jacobi process with jumps with state space $E_V=[V_I,V_E]$ defined in  \eqref{eq:gen_nontransf}.
As mentioned above, according to the model, the spikes are generated when the process $X$ crosses a voltage threshold $V_I<S<V_E$ for the first time, that is at time $T_S=T_S(X)$. After the spike, the process is reset instantaneously to the starting position $V_I<x<V_E$, ready to start its evolution over again.
 This renewal condition guarantees that  the inter-spike intervals, i.e the time between two consecutive spikes, are
independent and all identically distributed as the first inter-spike interval $T_S$.
Proposition \ref{prop:intD} guarantees that we can consider equivalently
the Jacobi process with jumps $Y$ with state space $[0,1]$ starting at
$y=h(x)=(x-V_I)/(V_E-V_I)$ in the presence of the threshold $a=h(S)=(S-V_I)/(V_E-V_I)$ and other parameters defined in \eqref{eq:par}. This enables us to use the mathematical results obtained in Section \ref{section3}.
 For these reasons the quantity of interest in the mathematical analysis of the neuronal activity is the first passage time $T_a=T_a(Y)$.
The probability of firing (i.e.~the probability that the process $Y$  crosses the boundary $a$ within a finite time) is given by  \eqref{Lapl_bo3} for $q=0$.  Under  hypothesis \eqref{ass:1},
the crossing of $a$ occurs almost surely in finite time.

Furthermore, it is interesting, for the analysis of the firing activity, to study the first moment of $T_a$.
In fact, it is assumed that neurons express information about their input mainly by means of the average frequency of spikes described by the neuronal firing rate.
It can be mathematically defined in several different ways \cite{firing_rate}, here we choose the classical definition of the instantaneous firing rate as the reciprocal of the mean first passage time.

We distinguish between three possible regimes to characterize the neuronal activity.
If the asymptotic mean membrane potential is larger than the firing threshold $a$, then the process is in the so-called suprathreshold regime.
In the classical case, in this regime, the spikes are regular and the dynamics is driven mainly by the drift part.
If the asymptotic mean membrane potential is smaller than $a$, then the process is said to be in the subthreshold regime, and the noise plays a prominent role for the crossing of the threshold. Finally, if the asymptotic mean is equal to $a$, the process is said to be in the threshold regime.
The classical Jacobi process is in the suprathreshold regime for $$\mo>a\lam$$
 whereas the analogous condition for the Jacobi with jumps, using \eqref{eq:mom-bpn_1}, is
\begin{equation}
\label{eq:supra}
\mo>a\lam+ \int_0^\infty e^{-r}\overline{\Pi}(r)dr.
\end{equation}
We observe that in \eqref{eq:supra} the asymptotic mean $\mo/\lam$ of the classical Jacobi has to exceed the threshold plus a term given by the downward jumps.

The  dependence of $\mathbb E_y[T_a]$ on $a$ and $y$ is the same as the classical Jacobi (and all other classical single neuron models).
$\mathbb E_y[T_a]$ decreases with the difference $a-y$ as it can be easily seen from the expression \eqref{ET_nonloc_jacobi}.

However, the dependence of $\mathbb E_y[T_a]$ on the inputs parameters $\nu_e$, $\nu_i$, $\h$ is non-trivial since the contribution of $\mu$ is hidden in the function $W_{\phi}$ that merges the contribution of the drift and the diffusion component.
To investigate it, we consider the following examples in which we choose a special form of the measure $\Pi$.

\subsection{Example}\label{sec:ex}
We consider a parametric family of non-local Jacobi operators for
which $\overline\Pi(r)=\int_r^\infty \Pi(du)=e^{-\alpha r}$,  $r > 0$, is of exponential type, that is $\Pi(dr)=\alpha e^{-\alpha r} dr$.
In particular, let $\alpha \geq 1$ and consider the integro-differential
operator $\calJ_\alpha$ given, from \eqref{eq:defJ}, by
\begin{eqnarray}
\label{gen_ex1}
\calJ_\alpha f(y) &=&\frac{\sigma^2}{2} y(1-y)f''(y)
-\left(\lam y-\mo \right)f'(y) -\int_0^1 (f(r)-f(y) )\frac{r^{\alpha}}{y^{\alpha+1}} dr,
\end{eqnarray}
Then $\calJ_\alpha$ is a non-local Jacobi operator with $\h = \int_1^\infty\overline\Pi(r)dr=1/\alpha$ and 
\begin{equation}
\label{phi_ex1}
\phi(u)=u+\frac{2}{\sigma^2}\left(\mu-\frac{1}{u+\alpha} \right)-1.
\end{equation}
Assumption \eqref{ass:1} is satisfied  whenever
\begin{eqnarray}\label{ass:1_alpha}
\frac{\sigma^2}{2}< \mo-\frac{1}{\alpha},
\end{eqnarray}
suggesting that the noise amplitude has to be smaller than in the classical case. The more is the contribution of the downward jumps (smaller values of $\alpha$) the higher is the risk that a large value of $\sigma$ can lead the process across the lower boundary, a condition that we want to avoid.
Under assumption \eqref{ass:1_alpha}, the first moment of $T_a$ for the  Jacobi process with jumps with generator \eqref{gen_ex1} is (see Appendix \ref{appendix:ex1})
\begin{eqnarray}
\label{ET_nonloc_jacobi_ex1}
\mathbb E_y[T_a]= \frac{2(\alpha+1)}{\sigma^2(k_++1)(k_{-}+1)}\left({}_4F_3(1,1,\alpha+2,\frac{2\lam}{\sigma^2};2,k_++2,k_-+2; a)a \right.\nonumber \\
\left.-{}_4F_3(1,1,\alpha+2,\frac{2\lam}{\sigma^2};2,k_++2,k_-+2; y)y \right)
\end{eqnarray}
where
\begin{equation}\label{k1k2}
k_\pm=\frac{1}{2}\left(\alpha+2\mo/\sigma^2-1\pm\sqrt{(\alpha+2\mo/\sigma^2-1)^2-4(2\alpha\mo/\sigma^2-\alpha-2/\sigma^2)}\right).
\end{equation}

We want to investigate the sensitivity of the mean FPT to a change in the input parameters $\mu$, $\lambda$, $\sigma^2$  and $\h$. The precise analysis requires the derivative of  generalized hypergeometric functions with respect to the relevant
parameters. To avoid lengthy calculation, we only show by plots the qualitative behavior using numerical evaluations in correspondence of physiologically realistic parameters chosen as in \cite{lanska94}.
In this case the firing regime is suprathreshold if
\begin{equation}\label{regime}
\mo>a\lam+\frac{1}{1+\alpha},
\end{equation}
we observe that the asymptotic mean $\mo/\lam$ of the classical Jacobi is decreased by a term given by the downward jumps. The result is that the asymptotic mean of the  Jacobi process with jumps increases with $\alpha$.
The reason lies in the shape of the distribution $\Pi$, see Fig.\ref{fig_alpha}-left. For small values of $\alpha$ there is a higher probability that $r$ takes large values with  corresponding large jumps. Conversely for large values of $\alpha$ the probability mass is concentrated around zero favoring small jumps.
The consequence is shown in Fig.\ref{fig_alpha}-right: the mean FPT decreases as $\alpha$ increases.
In Fig.\ref{fig_alpha}-right we also use a  discretization scheme for simulating sample paths developed recently in \cite{don_lanteri}. At each time-step of the algorithm, a value for $r$ is sampled from the distribution $\Pi(dr)$ and according to the survival probability of $\mathcal T$ described in Section \ref{section2}, a jump may occur. 
 Then the trajectory moves according to the diffusion, from state $e^{-rY_t}$ if there was a jump, otherwise from state $Y_t$.
 Between the jump epochs the dynamics of the constructed process are purely diffusive and are simulated using the Milstein's discretization method. The curve is obtained simulating $2\cdot 10^4$ sample paths of the Jacobi process with jumps for each of the $100$ values of $\alpha$ considered. For each simulation the FPT is recorded and for every value of $\alpha$ a value of the mean FPT is obtained. The simulation results are of course subject to numerical errors, mainly due to the choice of the time-discretization step (here $5 \cdot 10^{-4}$) and the relative low number of FPTs considered. On the one hand, this confirms the importance of an analytic result, and, on the other hand, supports the validity of formula \eqref{ET_nonloc_jacobi}.
In fact, it is known that discretization schemes overestimate the mean FPT since undetected threshold crossings may occur inside each discretization interval (see for instance \cite{GSZ}). 
 {We stress, anyway, that one must be careful in the evaluation of Eq.\eqref{ET_nonloc_jacobi} or Eq.\eqref{ET_nonloc_jacobi_ex1} since the involved functions become soon very large or very small as $n$ increases.
To be sure of our evaluations, we  applied different numerical approaches obtaining the same results as a guarantee of the correctness of our computations.}  
  In Fig.\ref{fig_alpha}-right, one also observes that the mean first passage time, as a function of $\alpha$, is nonincreasing, which is an illustration of  the comparison result provided in \eqref{eq:CET}. Indeed, one uses the fact that, for any $0<\alpha \leq \alpha_1$, the function $e^{-\alpha_1  r}=\overline{\Pi}_1(r)\leq \overline{\Pi}(r)= e^{-\alpha r}, r>0$.

\begin{figure}
    \centering
    \begin{subfigure}[t]{0.42\textwidth}
        \centering
        \includegraphics[width=\linewidth]{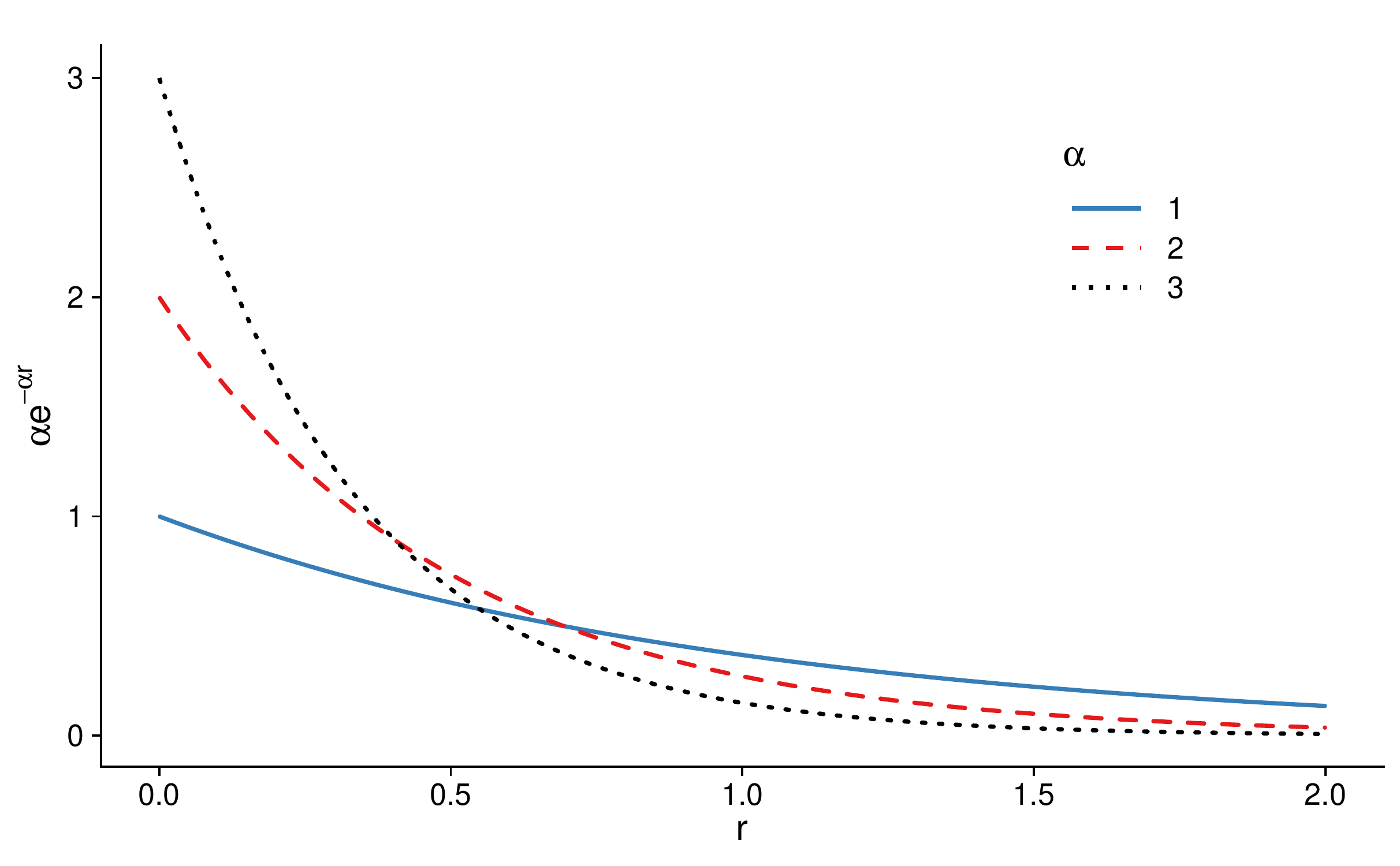}
        \caption{Function $\alpha e^{-\alpha r}$ for three values of $\alpha$ given in the legend.} \label{fig:alphaA}
    \end{subfigure}
    \hfill
    \begin{subfigure}[t]{0.50\textwidth}
        \centering
        \includegraphics[width=\linewidth]{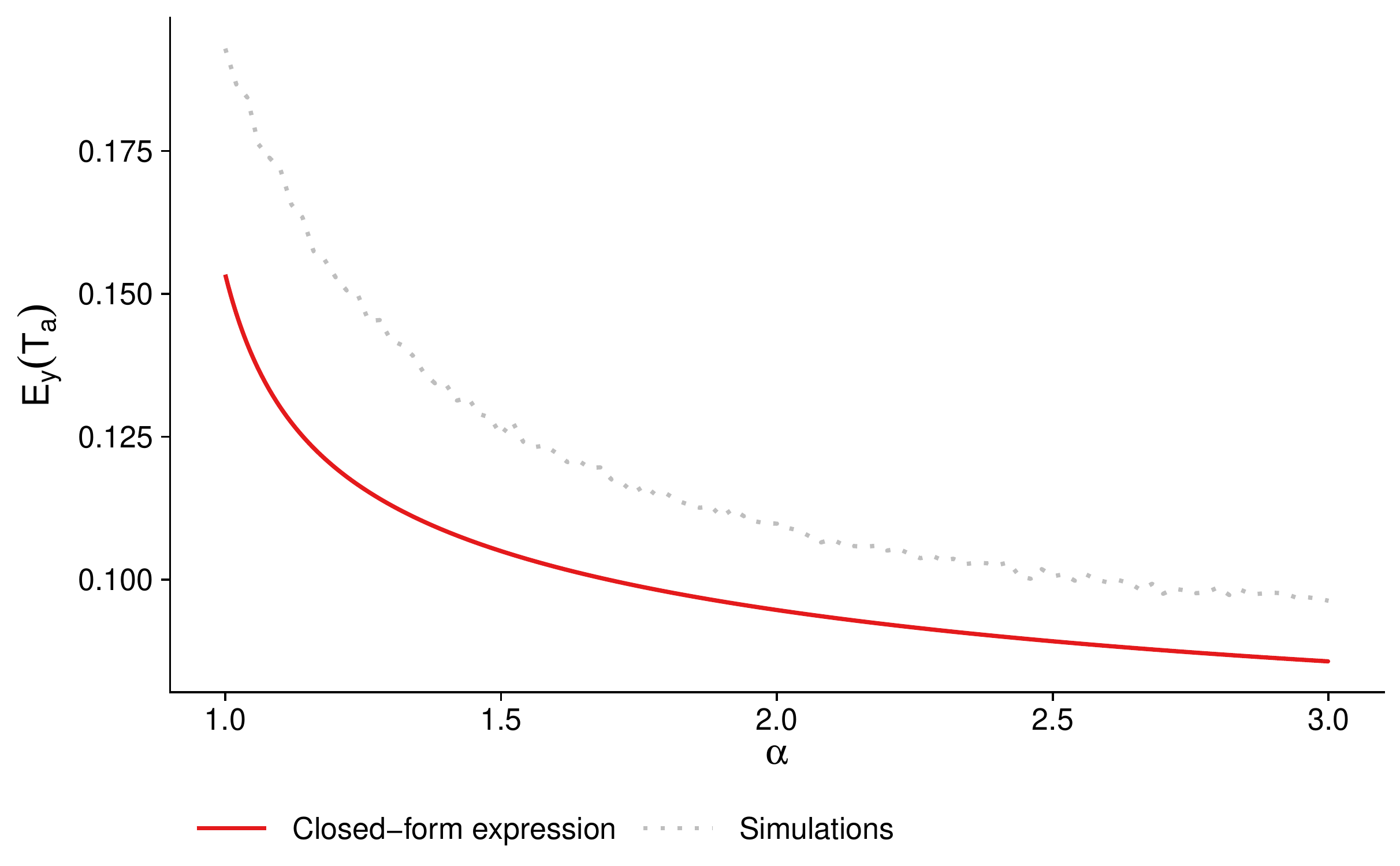}
        \caption{Mean FPT $\mathbb E_y[T_a]$ for the  Jacobi process with jumps from  \eqref{ET_nonloc_jacobi} with $\Pi(dr)=\alpha e^{-\alpha r} dr$,
as function of $\alpha$. {The series is truncated after $150$ terms, but it has been checked numerically that the quick convergence of the series guarantees a correct evaluation of the whole sum.}
The other parameters are  $V_I=-10$mV, $V_E=100$mV,  $S=10$ mV,
$x=0$ mV,
$\tau=15$ ms,
$a=0.18$ mV,
$y=0.09$ mV,
$e=0.5$,
$i=-1$, $\nu_i=1$ ms$^{-1}$, $\nu_e=2.8$ ms$^{-1}$,
$\sigma^2=0.5 $ ms$^{-1}$. In grey, $\mathbb E_y[T_a]$ obtained from simulations of $2\cdot 10^4$ FPTs for each of the $100$ values of $\alpha$ with time step $dt=5 \cdot 10^{-4}$. } \label{fig:alphaB}
    \end{subfigure}
    \caption{}
\label{fig_alpha}
\end{figure}

Let us now investigate how sensitive is $\mathbb E_y[T_a]$ to a change in the incoming input rates.
As expected we find that the mean FPT  decreases for stronger excitatory inputs and increases with the inhibitory inputs.
This dependence is clearly visible in the color change in the heatmap in Fig.\ref{fig_heat} where the excitatory and inhibitory inputs are tuned simultaneously. The blue lines are the contour plots, i.e., the couples  ($\nu_e$, $\nu_i$) that produce the same mean FPT.
The values of $\nu_e$ are chosen to meet condition \eqref{ass:1_alpha}
or equivalently
\begin{eqnarray}
\nu_e >\frac{1}{e}\left(\frac{V_I}{\tau(V_E-V_I)}+\frac{\sigma^2}{2}+\frac{1}{\alpha} \right).\label{ass:1_alpha_nue}
\end{eqnarray}
The  heatmaps are obtained from \eqref{ET_nonloc_jacobi}  with  $\Pi(dr)=\alpha e^{-\alpha r} dr$ and $\alpha=3$ (Fig.\ref{fig_heat} - left) and from  \eqref{ET_classical} (Fig.\ref{fig_heat} - right).
Alternatively one can evaluate \eqref{ET_nonloc_jacobi_ex1}
with the package ${\tt hypergeo}$ \cite{hypergeo} for
the software environment  ${\tt R}$.

We observe three main differences between the mean FPT of the two processes:
\begin{itemize}
\item in the non-local case, due to the presence of the term $1/\alpha$ in \eqref{ass:1_alpha_nue}, we need a larger excitatory input rate to guarantee a finite FPT,
\item for the same choices of parameters, the waiting time before the first spike in the classical case is shorter than in the non-local case,
\item the shape of the contour plots changes.
\end{itemize}
Regarding the third item, in the classical case, if we increase the inhibitory input rate $\nu_i$, then we have to increase linearly the excitatory input rate $\nu_e$ to get the same mean FPT.
In the non-local case the jump part comes into play breaking this tight coupling.

\begin{figure}[h]
\centering
\includegraphics[width=8 cm]{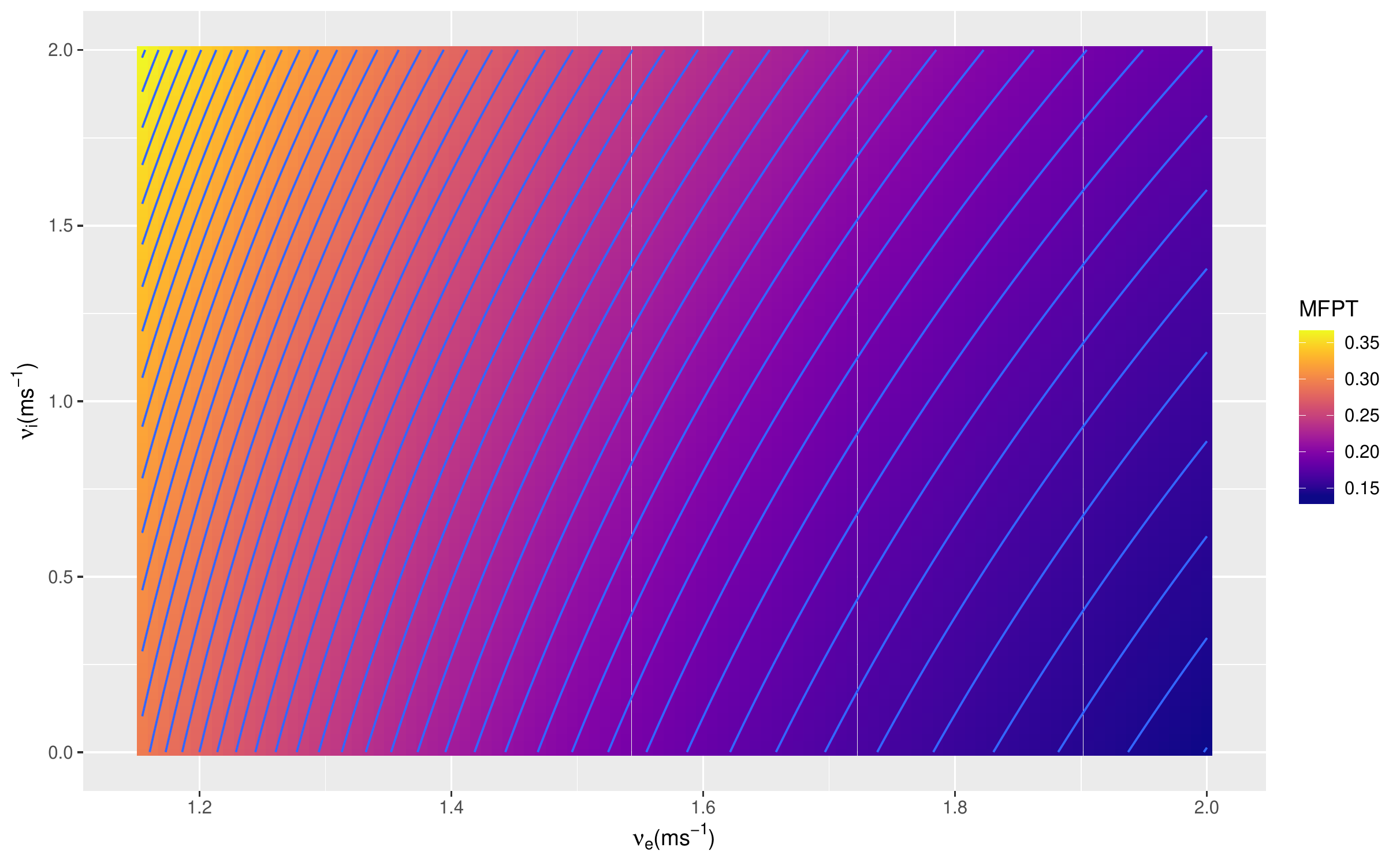}
\includegraphics[width=8 cm]{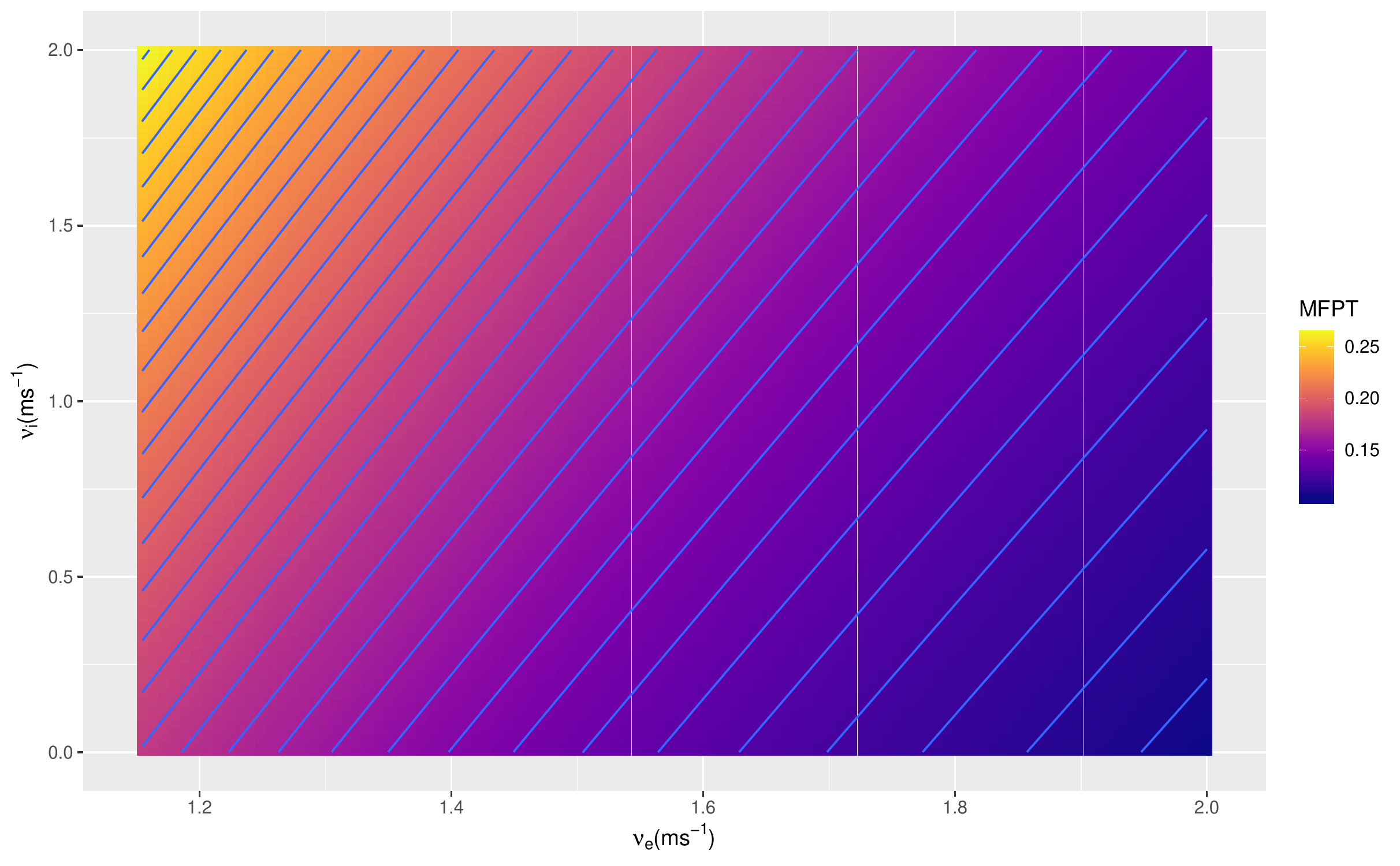}
\caption{Mean FPT, $\mathbb E_y[T_a]$, for the non-local (left) and classical (right) Jacobi processes  as a function of the excitatory  and inhibitory input rates $\nu_e$ and $\nu_i$. The  heatmaps are obtained from \eqref{ET_nonloc_jacobi}  with  $\Pi(dr)=\alpha e^{-\alpha r} dr$, $\alpha=3$ (left) and from \eqref{ET_classical} (right). The other parameters are chosen as in Fig.\ref{fig_alpha}.
}\label{fig_heat}
\end{figure}

Fig.\ref{fig_sigma} plots the mean FPT of the Jacobi process with jumps with infinitesimal generator \eqref{gen_ex1} as a function of $\sigma^2$.
\begin{figure}
    \centering
    \begin{subfigure}[t]{0.45\textwidth}
        \centering
        \includegraphics[width=\linewidth]{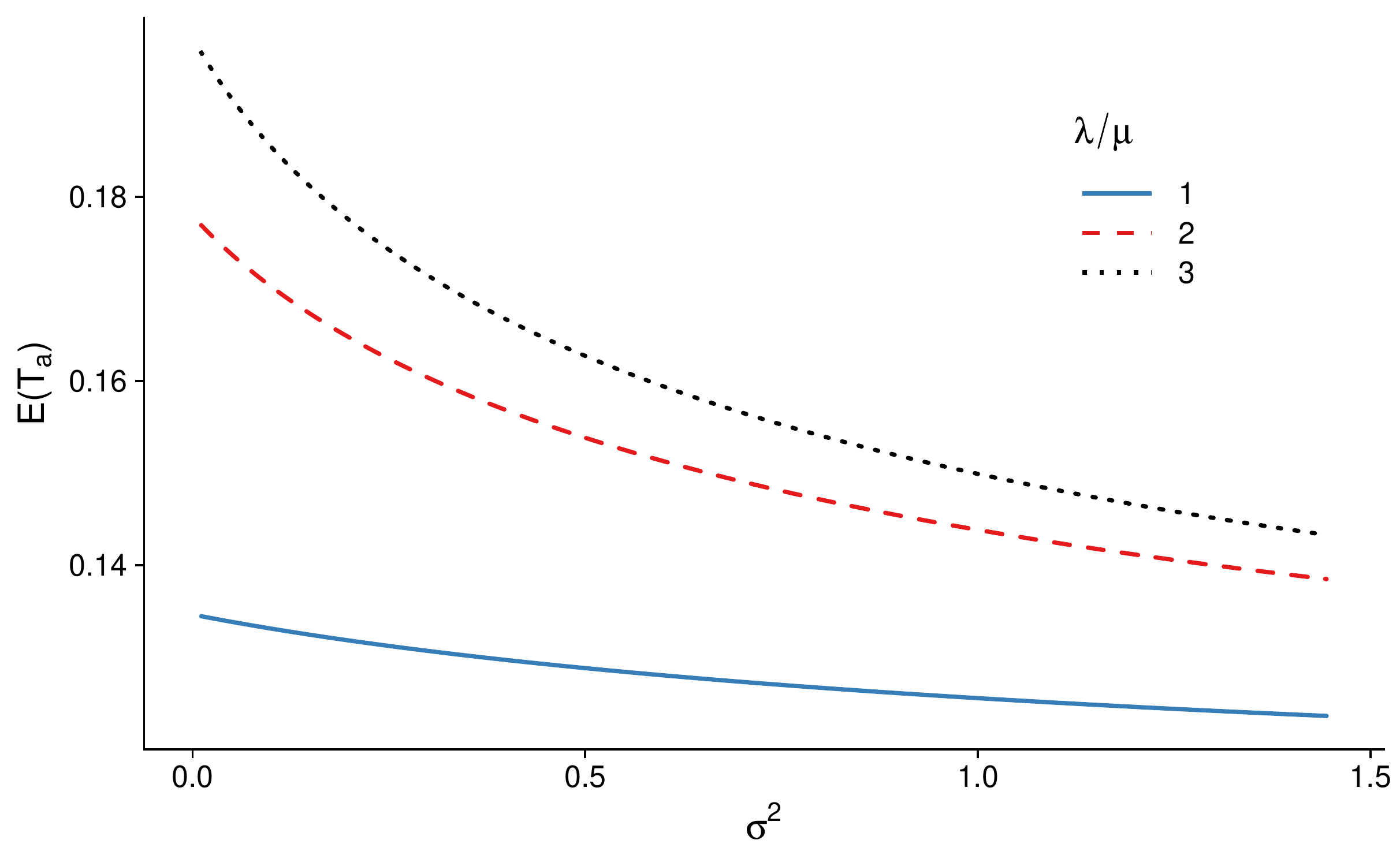}
        \caption{$\mathbb E_y[T_a]$ as a function of $\sigma^2$ for the  Jacobi process with jumps with infinitesimal generator \eqref{gen_ex1} for different values of the ratio $\lam/\mo$ given in the legend, using \eqref{ET_nonloc_jacobi}. In the plot $\alpha=3$, $\tau=15$ ms,
$a=0.18$ mV,
$y=0.09$ mV, $\nu_e=2.1$ ms$^{-1}$.} \label{fig:sigmaA}
    \end{subfigure}
    \hfill
    \begin{subfigure}[t]{0.45\textwidth}
        \centering
        \includegraphics[width=\linewidth]{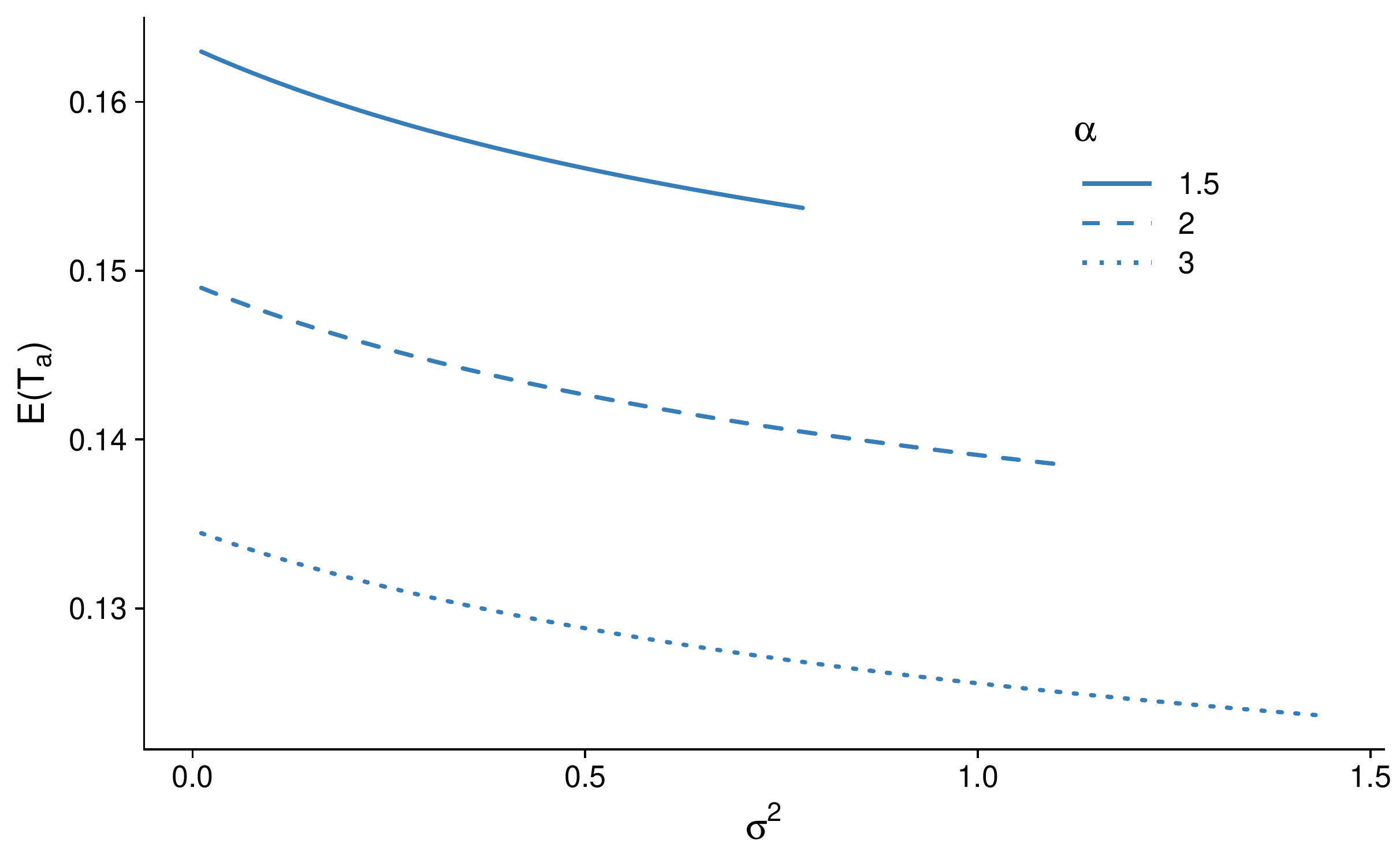}
        \caption{We consider the effect of $\alpha$ in the case $\lam \sim \mu$. All curves are plotted as function of $\sigma^2$ to meet assumption \eqref{ass:1_alpha}, and this is the reason why some lines stop before others.} \label{fig:sigmaB}
    \end{subfigure}
    \caption{}
\label{fig_sigma}
\end{figure}
As in the classical Jacobi model, $\mathbb E_y[T_a]$ decreases as $\sigma^2$ increases. This result is generally explained noting that an increase of variability facilitates the boundary crossing.

Since a closed form formula for the variance of $T_a$ is not available, it is natural to look at the asymptotic variance of the process $Y$ to study the role of $\sigma^2$.
From \eqref{eq:mom-bpn} and \eqref{eq:mom-bpn_1} we calculate the asymptotic variance of $Y$, $\mathrm{Var} (Y_{\infty})$, as
\begin{eqnarray*}
\mathrm {Var} (Y_{\infty})&=&\bpsi[p_2]-\bpsi[p_1]^2 \\ 
&=&\frac{\left(\mo -\frac{1}{1+\alpha}\right)\left(\sigma^4+\sigma^2\left(\mo -\frac{1}{2+\alpha}\right)\right)-\mo^2-\frac{1}{(1+\alpha)^2}+\frac{2\mo\lam^2}{1+\alpha}-\frac{\mo^2\sigma^2}{\lam}-\frac{\sigma^2}{\lam(1+\alpha)^2}+\frac{2\lam\mo\sigma^2}{1+\alpha}}{\lam^2+\lam \sigma^2}
\end{eqnarray*}
and one can get that the derivative with respect to  $\sigma^2$ is positive.
Then $\mathrm{Var} (Y_{\infty})$ increases with $\sigma^2$ and the variability usually favors the crossing of the threshold,  explaining the result of Fig.\ref{fig_sigma}.

As a final remark, we look at the blue solid curve in Fig.\ref{fig_sigma} (A). All the curves are obtained keeping fixed $\nu_e=2.1$ ms$^{-1}$ and changing $\nu_i$ ($0.1;1.5;1.9$ ms$^{-1}$) to get different ratios $\lam/\mo$, as it is usually done in the classical case. The blue solid curve is obtained in the case of a very weak inhibition $\nu_i=0.1$ ms$^{-1}$, that is the reason why the mean FPT is smaller and the behavior is different from the other two cases. We observe in Fig.\ref{fig_sigma} (B) that the presence of the jump part compensates the absence of the inhibitory inputs, increasing the waiting time before the neuronal spike.

In Fig.\ref{fig_comparison}, we compare the firing rate, that is here the reciprocal of $\mathbb{E}_y[T_a]$, for the classical and the non-local Jacobi processes for the same choices of the common parameters. In the classical case a strong excitation rate $\nu_e$ produces an intense activity of the neuron that grows linearly with $\nu_e$. In the non-local case the value of the firing rate
is almost halved and shows a sub-linear growth with respect to $\nu_e$.

In Fig.\ref{fig_comparison} the vertical lines indicate the threshold regimes for the two dynamics, separating the subthreshold on the left from the suprathreshold on the right. We have used three colors to highlight the intervals of sub and suprathreshold for the two processes. We observe that the difference between the two firing rates is smaller in the subthreshold regime, whereas the gap increases in the suprathreshold regime where the dynamics of the classical Jacobi is mainly driven by the drift component, especially being $\nu_i$ much smaller than $\nu_e$.
Moreover, numerical evidences suggest that the firing rate for the  Jacobi process with jumps saturates, differently from the classical one (at least for this range of parameters).
A similar kind of saturation is observed in the classical case, but only in the presence of a non-zero refractory period, see for instance Fig.2 of \cite{longtin_bulsara}.

\begin{figure}[h]
\centering
\includegraphics[width=10 cm]{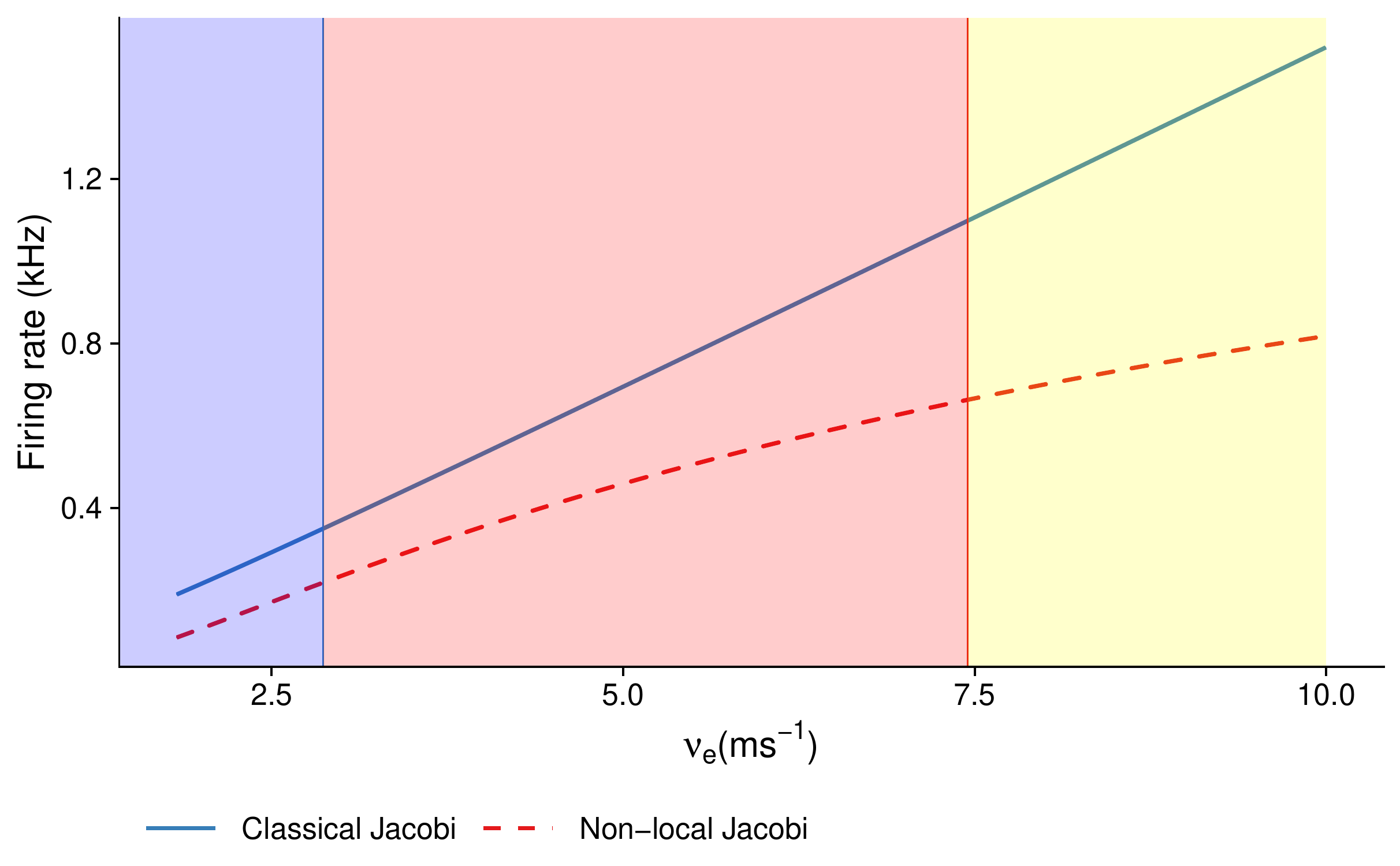}
\caption{Firing rate $1/\mathbb E_y[T_a]$ for classical and non-local Jacobi processes as function of the excitatory input rate $\nu_e$. Curves are obtained from  \eqref{ET_nonloc_jacobi} and \eqref{ET_classical} for
$y=0.09$ mV, 
$a=(S-V_I)/(V_E-V_I)=0.75$ mV, $\alpha=3$,
$\nu_i=0.2$ ms$^{-1}$,
$\tau=5$ ms,
$e=0.2$,
$i=-0.2$,
$\mu_e=e\nu_e$,
$\mu_i=i\nu_i$,
$\sigma^2=0.1 $ ms$^{-1}$. The firing rate is reduced by the downward jumps. The vertical lines indicate the threshold regimes for the two dynamics, separating the subthreshold and the suprathreshold regimes.
The blue region corresponds to subthreshold regime for both processes, the red one corresponds to subthreshold for the non-local and suprathreshold for the classical Jacobi process and finally the yellow area represents suprathreshold regime for both processes.
 }\label{fig_comparison}
\end{figure}

Then, in the case of a strong excitatory input,  the presence of the jump part can contribute to reduce the firing rate.
We stress that we choose incoming input parameters that are up to $10$ times stronger than those of an healthy neuron, see for instance physiological parameter values chosen in \cite{lanska94},  to illustrate instances in which anomalous behaviors arise.
We speculate, that one can  refine this model
to describe a pharmacological treatment of neurons whose activity is too intense, like in epileptic seizures or eventually to
 model the effect of  drug consumption.

\subsubsection{Example: A special case}
 Let us consider the previous example in the special case of an input dependent distribution $\overline\Pi$, in particular let, with $\delta=\mu-1$, $\overline\Pi(r)=e^{-\delta r}$,  $r > 0$,  and,  $\sigma^2=1$.

Let $\delta >1$ and consider the integro-differential
operator $\calJ_\delta$ given by \eqref{gen_ex1} for $\alpha=\delta$.
One gets that in this case $\h=1/\delta$ and
\begin{equation}
\phi(u)=u+2\left(\mu-\frac{1}{u+\mu-1} \right)-1=
\frac{(u+\mo)(u+\mo -2)}{u+\mo -1}.
\end{equation}
Since $\mu=\delta+1> 2$, $\phi(0)=\frac{\mo(\mo -2)}{\mo -1}>0$, the required assumption \eqref{eq:cond_phi} is satisfied.

In this case,  the distribution  $\overline\Pi$ being dependent on the incoming excitatory inputs, we have that the contribution of the jump part reduces as $\nu_e$ increases.
This means that if the excitatory input is strong then the neuron fires with a weak contrast of the jump part, whereas if the input is weak and the potential is far from the threshold then the jump component tends to make the neuron silent. This behavior avoids unnecessary spikes and enhances the information transmission.
This case may describe the situation in which inhibitory neurons inside the network, that are regulatory for the neuron activity, are not able to oppose to an increment in the excitatory inputs that may lead to an excessive spiking activity of the neuron under study.

On the contrary, if one wants to extend the model
to address a pharmacological treatment of neurons whose activity is too intense we suggest to choose some jump distribution that depends on the inverse of $\mu$ or in general a heavy-tailed distribution that favors the large jumps reducing consistently the firing activity.

As a future work we plan to investigate the effects of other distributions of the jumps,  with special attention to
heavy-tailed distributions that favors large jumps.
Moreover it would be interesting to add also upward jumps to the model and investigate the case of a signal dependent noise as in \cite{greenwood_lan},\cite{lansky_sac} and \cite{longtin_bulsara}, to study the possible role of the inhibitory jumps in improving the information transmission through a coherence resonance between the input and the output.

\section{Conclusions}

The contribution of this paper is twofold. On the one hand,  it allows to advance  in the LIF modeling in  the mathematical neuroscience context. On the other hand, it also contains an original  methodology, based on intertwining relationship, to study the classical first passage time problem of a Markov process with possible jumps.

Starting from the idea of endowing the classical LIF model with features that are more in keeping with the phenomenological reality we introduce a diffusion process with jumps for the description of the activity of a single neuron.
Among the strengths of the presented model we have that on one hand the {\em good} properties of the classical Jacobi process are preserved: the state space is limited and the frequency and the amplitude of the jumps  are state-dependent. On the other hand it also accounts for inputs that prevent the diffusion limit due to their amplitudes and/or to their frequencies. 
{In this way one can assign different weights of the incoming inputs depending on whether they arrive more or less close to the trigger zone, as a first attempt to consider the neuron not only as a point.}
 Moreover, these downward jumps are able to reduce the firing rate and introduce saturation even in the absence of a refractory period.
This constitutes a novelty compared to other LIF models for which the firing rate increases linearly (and unbounded).
The feature that the model accounts possibly large downward jumps suggests its use for describing the contribution of a strong internal inhibitory input or the effect of an external factor, like a pharmacological treatment or the intake of drugs and alcohol that interfere with the standard activity of the neuron. The tuning of this quantity can also help the  investigation of the role of inhibition in the information transmission.
 {Moreover, the fact that the jumps are more frequent close to the inhibitory reversal potential could describe the phenomenon of neuronal accommodation (see for instance \cite{baker1989}). In this framework after a spike, if a current which rises sufficiently slowly is applied, it will never evoke an action potential  until the inactive phase is over.}
Finally, the high degree of freedom in the choice of the jump distribution and the relatively easy numerical implementation permit the description of multiple different situations.

As far as the mathematical novelty is concerned, we note that for the study of the firing rate, we had to develop the study of the first-passage times of the proposed Markovian Jacobi process with jumps through a constant boundary.
We are able to inherit some results from the classical Jacobi process to the process with jumps, thanks to a general strategy, original in the context of FPT problems, that relies on intertwining relations between the semigroups of the classical Jacobi process and its generalization.
Therefore, this paper provides an additional application of such a concept in the theory of Markov processes by transferring  $q$-invariant functions from a reference semigroup to semigroups that are in its intertwining orbit. This new approach enables us to characterize the Laplace transform of the FPT, expressed in terms of a generalization of the Gauss hypergeometric function that we introduce. As by-product, we obtain  a closed-form expression for its expectation.
This result appears of particular interest since an exact simulation method for the paths of the process considered is not yet available. We also mention that, relying on the recent works \cite{Choi, Loef}, where a comprehensive fluctuation theory for skip-free Markov chains is established, one could exploit the intertwining relation to identify the Laplace transform of the first exit time from an interval.  More generally, intertwining relations enable to relate the set of $q$-invariant functions (martingales), and,
more generally, the convex cone of  $q$-excessive functions (supermartingales) between semigroups. In potential theory, $q$-excessive functions
are well-known to characterize the Laplace transform of the first passage time of a set for the  processes.
For this standpoint, the intertwining approach seems to be a natural and promising way to deal with the first exit time problem for a Markov process with two sided-jumps. These  will be the subject of future investigations.

Finally, we stress that, despite our application in the context of mathematical neuroscience, the results on the Jacobi process with jumps and its first passage time through a constant boundary are novel and of a general nature. In particular, we mention the use of the Jacobi process in the context of population genetics and mathematical finance, where it often goes under the name of Wright-Fisher diffusion.

\appendix
\section{First passage times of the classical Jacobi process}
\label{appendix:class_jac}
Let $Y=(Y_t)_{t\geq 0}$ be the Jacobi process with infinitesimal generator given, for a smooth  function $f$ on $[0,1]$, by
\begin{equation}\label{inf_mom_jacobi}
\calJd_{\mu} f(y)= \frac{\sigma^2}{2} y (1-y)f''(y)-\left(\lambda y-\mo\right)f'(y)
\end{equation}
with $\mu > \sigma^2/2$ to ensure that $0$ is an entrance boundary \cite{pearson}.  We recall that, from \eqref{def:bern}, $W_{\phi}(n+1)= (\phi(0)+1)_n=\left(\frac{2\mu}{\sigma^2}\right)_n$,
and, thus, in this case \[ {}_2F_1\left(a,b;\phi; y\right)= {}_2F_1\left(a,b;\phi(0)+1;y\right)\]  the latter being  the Gauss hypergeometric function.
In what follows, we recall the expression  of the Laplace transform and the first moment of its first passage time, which can be found in \cite{lanska94}, see also  \cite{don_jacobi}.
\begin{proposition}
Let $0<y<a<1$,  the Laplace transform of the first passage time
\begin{equation}\label{eq:def_T}
  T_a = \inf\{t>0;\: Y_t\geq a\}
\end{equation}
of the Jacobi process \eqref{inf_mom_jacobi}
is given by
\begin{equation}
\label{Lapl_class}
\mathbb{E}_{y}[e^{-q T_a}\mathbb{I}_{\{T_a<\infty\}}]=\frac{{}_2F_1\left(\kappa(q),\theta(q);\frac{2\mu}{\sigma^2};y\right)}
{{}_2F_1\left(\kappa(q),\theta(q),\frac{2\mu}{\sigma^2};a\right)}
\end{equation}
where $k(q)$ and $\theta(q)$ are solution of the system
\begin{equation}\label{param_lapl}
\kappa(q)+\theta(q)+1=\frac{2\lam}{\sigma^2}, \quad \kappa(q)\theta(q)=\frac{2q}{\sigma^2} 
\end{equation}
More specifically,
 \begin{equation}
\kappa(q)=\frac{2 q}{\theta(q) \sigma^2}, 
 \nonumber
\end{equation}
and
 \[ \theta(q)=\frac{2\lam -\sigma^2\pm\sqrt{(\sigma^2-2 \lam)^2-8q \sigma^2}}{2 \sigma^2}=\bar\lam\pm\sqrt{\left|\overline{\lam}^2-\frac{2q}{\sigma^2}\right|}\left(\mathbb{I}_{\{q\leq \frac{\sigma^2\overline{\lam}^2}{2}\}}+i\mathbb{I}_{\{q> \frac{\sigma^2\overline{\lam}^2}{2}\}}\right).\]
where $\overline\lam=\frac{\lam}{\sigma^2}-\frac12 \geq 0$, as, by condition \eqref{ass:1}, $\lambda>\mu>\sigma^2/2$. 
Finally, the first moment of $T_a$ is
\begin{eqnarray}
\label{ET_classical}
\mathbb E_{y}[T_a]&=&\frac{1}{\mo}\sum_{n=0}^{\infty}\frac{\left(\frac{2\lam}{\sigma^2}\right)_n}{\left(\frac{2\mo}{\sigma^2}+1\right)_n}\frac{a^{n+1}-y^{n+1}}{n+1} \\
&=&\frac{1}{\mu}\left({}_3F_2\left(1,1,\frac{2\lam}{\sigma^2};2,\frac{2\mo}{\sigma^2};a\right)a
-{}_3F_2\left(1,1,\frac{2\lam}{\sigma^2};2,\frac{2\mo}{\sigma^2};y\right)y\right).
\end{eqnarray}
  \end{proposition}

\section{Mean of the first passage time of the Jacobi process with exponential jumps type}
\label{appendix:ex1}

\begin{proposition}
Under the condition $\frac{\sigma^2}{2}< \mo-\frac{1}{\alpha}$, the first moment of $T_a$ for the  Jacobi process with jumps with generator \eqref{gen_ex1} is, for any $0<y<a$,
\begin{eqnarray}
\label{ET_nonloc_jacobi_ex1_append}
\mathbb E_y[T_a]= \frac{2(\alpha+1)}{\sigma^2(k_++1)(k_{-}+1)}\left({}_4F_3(1,1,\alpha+2,\frac{2\lam}{\sigma^2};2,k_++2,k_-+2; a)a \right.\nonumber \\
\left.-{}_4F_3(1,1,\alpha+2,\frac{2\lam}{\sigma^2};2,k_++2,k_-+2; y)y \right)
\end{eqnarray}
where
\begin{equation}\label{k1k2}
k_\pm=\frac{1}{2}\left(\alpha+2\mo/\sigma^2-1\pm\sqrt{(\alpha+2\mo/\sigma^2-1)^2-4(2\alpha\mo/\sigma^2-\alpha-2/\sigma^2)}\right).
\end{equation}
\end{proposition}
\begin{proof}
For $\h =1/\alpha$, $\alpha \geq 1$,  the Bernstein function $\phi$ defined in \eqref{phi_ex1}
can be written as
\begin{equation}
\phi(u)=\frac{(u+k_+)(u+k_-)}{(u+\alpha)}
\end{equation}
with $k_+$ and $k_-$ are defined in \eqref{k1k2}.
This implies that, for any $n\geq0$,
\begin{equation}
W_\phi(n+2)=\frac{(k_++1)_{n+1}(k_-+1)_{n+1}}{(\alpha+1)_{n+1}}.
\end{equation}
Using that $(k_\pm+1)_{n+1}=(k_\pm+1)(k_\pm+2)_{n}$ and  \eqref{ET_nonloc_jacobi}, one gets  that
\begin{equation}
\label{ET_nonloc_jacobi_prop}
\mathbb E_y[T_a]= \frac{2(\alpha+1)}{\sigma^2(k_++1)(k_-+1)} \sum_{n=0}^{\infty}\frac{(1)_n(1)_n(\frac{2\lam}{\sigma^2})_n(\alpha+2)_{n}}{(2)_n(k_++2)_{n}(k_-+2)_{n}}\frac{a^{n+1}-y^{n+1}}{n!}.
\end{equation}
Finally, the definition of the generalized hypergeometric function proves \eqref{ET_nonloc_jacobi_ex1_append}.
\end{proof}

\section*{Acknowledgments}
The authors are indebted to the associated editors and an anonymous  referee for valuable and constructive comments that improved the presentation of the paper. G.D. and L.S. have been partially supported by the MIUR-PRIN 2022 project \lq\lq Non-Markovian dynamics and non-local equations\rq \rq, no. 202277N5H9.

\bibliographystyle{abbrv}

\end{document}